\documentclass[final,leqno]{siamltex}

\usepackage{amsfonts}
\usepackage{amsmath}
\usepackage{amssymb}
\usepackage{mathtools}

\providecommand{\U}[1]{\protect\rule{.1in}{.1in}}
\providecommand{\U}[1]{\protect\rule{.1in}{.1in}}

\newtheorem{assumption}[theorem]{Assumption}
\newtheorem{remark}[theorem]{Remark}
\newtheorem{example}[theorem]{Example}

\newcommand{\e}{\mathrm{e}}

\renewcommand{\d}{\mathrm{d}}
\newcommand{\eps}{\varepsilon}
\newcommand{\E}{\mathbf{E}}
\newcommand{\R}{\mathbb{R}}
\newcommand{\cF}{\mathcal{F}}
\newcommand{\cP}{\mathcal{P}}
\newcommand{\cD}{\mathcal{D}}

\newcommand{\lp}{\prescript{{\bf p}}{}}

\newcommand{\state}{{y}}

\begin{document}

\title{A Maximum Principle for Optimal Control of Stochastic Evolution Equations}

\author{Kai~Du\thanks{Department of Mathematics, ETH Zurich, 8092 Zurich, Switzerland
({\tt kai.du@math.ethz.ch}). Financial support by the National Centre of Competence
in Research ``Financial Valuation and Risk Management'' (NCCR FINRISK),
Project D1 (Mathematical Methods in Financial Risk Management) is gratefully
acknowledged. The NCCR FINRISK is a research instrument of the Swiss National
Science Foundation.}
\and Qingxin~Meng\thanks{Department of Mathematical Sciences, Huzhou University,
Zhejiang 313000, China ({\tt mqx@hutc.zj.cn}). Financial support by the National
Natural Science Foundation of China (Grant No. 11101140, No.11301177), China
Postdoctoral Science Foundation (2011M500721, 2012T50391), and the Natural
Science Foundation of Zhejiang Province (Grant No. Y6110775, Y6110789) is
gratefully acknowledged.}}
\maketitle

\begin{abstract}
A general maximum principle is proved for optimal controls of abstract semilinear
stochastic evolution equations. The control variable, as well as linear
unbounded operators, acts in both drift and diffusion terms, and the control
set need not be convex.
\end{abstract}



\pagestyle{myheadings}
\thispagestyle{plain}
\markboth{KAI DU AND QINGXIN MENG}{STOCHASTIC MAXIMUM PRINCIPLE}

\section{Introduction}

In this paper, we study an abstract infinite-dimensional stochastic control problem
whose purpose is to minimize a cost functional
\[
J(u)=\mathbf{E}\int_{0}^{1}l(t,x(t),u(t))\,\d t+\mathbf{E}h(x(1))
\]
subject to the semilinear stochastic evolution equation (SEE)
\begin{equation}\label{eq:state}
\bigg\{\begin{aligned}
\mathrm{d}x(t)  &  =[A(t)x(t)+f(t,x(t),u(t))]\mathrm{d}
t+[B(t)x(t)+g(t,x(t),u(t))]\,\mathrm{d}W_{t},\\
x(0)  &  =x_{0},
\end{aligned}
\end{equation}
where the state process $x(\cdot)$ and control $u(\cdot)$ take values in
infinite dimensional spaces, $A(t)$ and $B(t)$ are both random unbounded
operators, $f,g,l$ and $h$ are given random functions, and $W$ is a standard
Wiener process.

A classical approach for optimal control problems is to derive necessary
conditions satisfied by an optimum, such as Pongtryagin's maximum principle
(cf. \cite{pontryagin1962mathematical}). Since 1970s, the maximum principle was
extensively studied for stochastic control systems: in the finite dimensional
case, it has been solved by Peng \cite{peng1990general} in a general setting
where the control was allowed to take values in a nonconvex set and enter into
the diffusion, while in the infinite-dimensional case, the existing literature,
e.g. \cite{bensoussan1983stochastic,
hu1990maximum,tang1993maximum,tudor1989optimal,zhou1993necessary}, required at
least one of the three assumptions that 1) the control domain was convex, 2)
the diffusion did not depend on the control, and 3) the state equation and cost
functional were both linear in the state variable. So far the general maximum
principle for infinite-dimensional stochastic control systems, i.e., the
counterpart of Peng's result, remained open for a long time. In this paper we
attempt to fill this gap.

In view of the second-order variation method developed by Peng
\cite{peng1990general}, the main difficulty in the infinite-dimensional case
lies in the step of second-order duality analysis, i.e., analyzing the
quadratic term in a variational inequality, which, in the finite-dimensional
case, can be worked out by means of the fact that the auxiliary process, called
the second-order adjoint process, satisfies a well-solved backward stochastic
differential equation (BSDE). However, in the case here, the corresponding BSDE
is operator-valued which seems rather difficult to solve. In this paper, we
develop a new procedure to do the second-order duality analysis: by virtue of
the Lebesgue differentiation theorem and an approximation argument, we prove
that, when the perturbation tends to zero, the quadratic term in the
variational inequality converges to a bilinear functional which can be
represented by an operator-valued process --- this is the second-order adjoint
process. Our approach, we believe, could work not only in
the abstract framework but also in many concrete cases.

Very recently, two other works \cite{lvzhangmaximum,fuhrman2012stochastic}
besides ours were also concerned with the general stochastic maximum principle
in infinite dimensions, while the three ones differ in both forms of state
equations and key approaches. The preprint \cite{lvzhangmaximum}, within an
abstract framework, focused on how to solve the operator-valued BSDE properly,
and to this end, introduced a notion called ``relaxed transposition solution''.
Their approach and result looked restrictive due to some technical assumptions.
In \cite{fuhrman2012stochastic} as well as its long version \cite{tessitore2013stochastic}, the authors considered a concrete stochastic
parabolic PDE with deterministic coefficients. The approach there, including a
compactness argument, required the Markov structure of the system.

The rest of this paper is organized as follows. In Section 2, we formulate the
problem in abstract form, and state our main results. Section 3 is devoted
to derive the basic $L^{p}$-estimate for SEEs. In Section 4, we study the
representation and properties of a stochastic bilinear function. The maximum
principle is proved in Section 5. In the final section, we discuss two examples.

We finish the introduction with several notations. Let
$(\Omega,\mathcal{F},\mathbb{F},\mathbf{P})$ be a filtered probability space
with the filtration $\mathbb{F}=(\mathcal{F}_{t})_{t\geq0}$ generated by a
1-dimensional\footnote{This restriction is just for simplicity. The approaches and results in
this paper can be extended without any nontrivial difficulty to the system
driving by a multi-dimensional Wiener process.} standard Wiener process $\{W_{t};t\geq0\}$ and
satisfying the usual conditions. Denote by $\mathcal{P}$ the predictable
$\sigma$-field. Let $H$ be a separable real Hilbert space and $\mathcal{B}(H)$
be its Borel $\sigma$-field, $p\in\lbrack1,\infty]$. The following classes of
processes will be used in this article.

$\bullet\ \ L^{p}(\Omega\times[0,1],\cP,H)$\thinspace: the space of equivalence
classes of $H$-valued $\mathcal{F}\times\mathcal{B}([0,1])$-measurable processes
$x(\cdot)$ admitting a predictable version such that
$\mathbf{E}\int_{0}^{1}\Vert x(t)\Vert_{H}^{p}\mathrm{d}t<\infty$.

$\bullet\ \ L_{\mathbb{F}}^{p}(\Omega,C([0,1],H))$\thinspace: the space of
$H$-valued adapted processes $x(\cdot)$ with continuous paths such that
$\mathbf{E}\sup_{t\in\lbrack0,1]}\Vert x(t)\Vert_{H}^{p}<\infty$. Elements of
this space are defined up to indistinguishability.

$\bullet\ \ L_{\mathbb{F}}^{p}([0,1],L^{p}(\Omega,H))$\thinspace: the space of
equivalence classes of $H$-valued adapted processes $x(\cdot)$ such that
$x(\cdot):[0,1]\rightarrow L^{p}(\Omega,H)$ is $\mathcal{B}([0,1])$-measurable
and $\int_{0}^{1}\mathbf{E}\Vert x(t)\Vert_{H}^{p}\,\mathrm{d}t<\infty$.

$\bullet\ \ C_{\mathbb{F}}([0,1],L^{p}(\Omega,H))$\thinspace: the space of
$H$-valued adapted processes $x(\cdot)$ such that
$x(\cdot):[0,1]\rightarrow L^{p}(\Omega,H)$ is strongly continuous and
$\sup_{t\in\lbrack0,1]}\mathbf{E}\Vert x(t)\Vert_{H}^{p}<\infty$. Elements of
this space are defined up to modification.

Moreover, denote by $\mathfrak{B}(H)$ the Banach space of all bounded linear
operators from $H$ to itself endowed with the norm $\Vert
T\Vert_{\mathfrak{B}(H)}=\sup\{\Vert Tx\Vert_H:\Vert x\Vert_{H}=1\}$. We shall
define the following spaces with respect to $\mathfrak{B}(H)$-valued processes
and random variables.

$\bullet\ \
L_{\mathrm{w}}^{p}(\Omega\times\lbrack0,1],\cP,\mathfrak{B}(H))$\thinspace: the
space of equivalence classes of $\mathfrak{B}(H)$-valued processes $T(\cdot)$
such that $\langle x,T(\cdot)y\rangle\in L^{p}(\Omega
\times\lbrack0,1],\cP,\mathbb{R})$ for any $x,y\in H$, and $\mathbf{E}\int_{0}
^{1}\Vert T(t)\Vert_{\mathfrak{B}(H)}^{p}\mathrm{d}t<\infty$. Here the
subscript \textquotedblleft$\mathrm{w}$\textquotedblright\ stands for
\textquotedblleft weakly measurable\textquotedblright.

$\bullet\ \
C_{\mathbb{F}}([0,1],L_{\mathrm{w}}^{p}(\Omega,\mathfrak{B}(H)))$\thinspace:
the space of $\mathfrak{B}(H)$-valued processes $T(\cdot)$ such that
$\langle x,T(\cdot)y\rangle\in\mathcal{C}_{\mathbb{F}}([0,1],L^{p}%
(\Omega,\mathbb{R}))$ for any $x,y\in H$, and $\sup_{t\in\lbrack
0,1]}\mathbf{E}\Vert T(t)\Vert_{\mathfrak{B}(H)}^{p}<\infty$. Elements of this
space are defined up to modification.

$\bullet\ \
L_{\mathrm{w}}^{p}(\Omega,\mathcal{F}_{t},\mathfrak{B}(H))$\thinspace: the
space of equivalence classes of $\mathfrak{B}(H)$-valued random variable $T$
such that $\langle x,Ty\rangle\in L^{p}(\Omega,\mathcal{F}_{t},\R)$ for any
$x,y\in H$, and $\mathbf{E}\Vert T\Vert_{\mathfrak{B}(H)}^{p}<\infty$.

\section{Formulation and main results}

\subsection{Problem formulation}

Let $H$ and $V$ be two separable real Hilbert spaces such that $V$ is densely
embedded in $H$. We identify $H$ with its dual space, and denote by $V^{\ast}$
the dual of $V$. Then we have $V\subset H\subset V^{\ast}$. Denote by
$\left\Vert \cdot\right\Vert _{H}$ the norms of $H$, by $\left\langle
\cdot,\cdot\right\rangle $ the inner product in $H$, and by $\left\langle
\cdot,\cdot\right\rangle _{\ast}$ the duality product between $V$ and
$V^{\ast}$. The notation $\mathfrak{B}(X;Y)$ stands for the usual
Banach space of all bounded linear operators from Banach space $X$ to Banach
space $Y$, and simply $\mathfrak{B}(X)=\mathfrak{B}(X; X)$.

Now we consider the controlled stochastic evolution system \eqref{eq:state} in
an abstract way:
\begin{equation*}
\bigg\{\begin{aligned}
\mathrm{d}x(t)  &  =[A(t)x(t)+f(t,x(t),u(t))]\mathrm{d}
t+[B(t)x(t)+g(t,x(t),u(t))]\,\mathrm{d}W_{t},\\
x(0)  &  =x_{0},
\end{aligned}
\end{equation*}
with the control process $u(\cdot)$ taking values in a set $U$, given
stochastic
evolution operators%
\[
A:\Omega\times[0,1]\rightarrow\mathfrak{B}(V; V^{\ast}%
)\ \ \ \ \text{and}\ \ \ \ B:\Omega\times[0,1]\rightarrow\mathfrak{B}%
(V; H),
\]
and nonlinear terms%
\[
f,\,g:\Omega\times[0,1]\times H\times U\rightarrow H.
\]
Here the \emph{control set} $U$ is a nonempty Borel-measurable subset of a
metric space whose metric is denoted by $\mathrm{dist}(\cdot,\cdot)$. Fix an
element (denoted by $0$) in $U$, and then define $|u|_{U}=\mathrm{dist}(u,0)$.
A $U$-valued predictable process $u(\cdot)$ is \emph{admissible} if there
exists a number $\delta_u>0$ such that
\[
\sup\big\{\mathbf{E}\left\vert u(t)\right\vert _{U}^{4+\delta_u}:t\in[0,1]\big\}<\infty.
\]
Denote by $U_{\mathrm{ad}}$ the set of all admissible controls.

Our optimal control problem is to find $u(\cdot)\in U_{\mathrm{ad}}$
to minimize the cost functional
\[
J(u(\cdot))=\mathbf{E}\int_{0}^{1}l(t,x(t),u(t))\,\mathrm{d}t+\mathbf{E}%
h(x(1))
\]
with given functions%
\[
l:\Omega\times[0,1]\times H\times U\rightarrow\mathbb{R}
\ \ \ \text{and\ \ \ }h:\Omega\times H\rightarrow\mathbb{R}.
\]

Throughout this paper, the SEE of form \eqref{eq:state} is read in a weak sense, i.e., for each $v\in V$,
\begin{align*}
\langle v, x(t)\rangle ~=~ & \langle v,x_0\rangle + \int_0^t \big[\langle v, A(t)x(t)\rangle_{*}+\langle v,f(t,x(t),u(t))\rangle\big]\,\mathrm{d}t\\
&+\int_0^1 \langle v, B(t)x(t)+g(t,x(t),u(t))\rangle \,\mathrm{d}W_{t}, \quad {\rm a.e.}~(\omega,t).
\end{align*}

We make the following assumptions. Fix some constants $\kappa\in(0,1)$ and
$K\in(0,\infty)$.

\begin{assumption}
\label{ass:onAB} The operator processes $A$ and $B$ are weakly predictable,
i.e., $\left\langle x,A(\cdot)y\right\rangle _{\ast}$ and $\left\langle
x,B(\cdot)y\right\rangle $ are both predictable processes for any $x,y\in V$; and
for each $(t,\omega)\in
\lbrack0,1]\times\Omega$,%
\begin{align*}
&  \left\langle x,A(t)x\right\rangle _{\ast}+\left\Vert B(t)x\right\Vert
_{H}^{2}\leq-\kappa\left\Vert x\right\Vert _{V}^{2}+K\left\Vert x\right\Vert
_{H}^{2},\\
&  \left\Vert A(t)x\right\Vert _{V^{\ast}}^{2}\leq K\left\Vert x\right\Vert
_{V}^{2},\ \ \ \ \ \ \ \forall x\in V.
\end{align*}
\end{assumption}
\begin{assumption}
\label{ass:onfglh} For each $(x,u)\in H\times U$, $f(\cdot,x,u)$,
$g(\cdot,x,u)$ and $l(\cdot,x,u)$ are all predictable processes, $h(x)$ is
$\mathcal{F}_{1}$-measurable random variable; for each
$(t,u,\omega)\in[0,1]\times U \times\Omega$, $f,g,l$ and $h$ are globally twice
Fr\`{e}chet differentiable with respect to $x$. The functions $f_{x}
,g_{x},f_{xx},g_{xx},l_{xx},h_{xx}$ are continuous in $x$ and dominated by the
constant $K$; $f,g,l_{x},h_{x}$ are dominated by $K(1+\Vert x\Vert_{H}+|u|_{U})
$; $l,h$ are dominated by $K(1+\Vert x\Vert_{H}^{2}+|u|_{U}^{2})$.
\end{assumption}
\begin{assumption}
\label{ass:onB} For each $(t,\omega)$,
\[
\left\vert \left\langle x,B(t)x\right\rangle \right\vert \leq K\left\Vert
x\right\Vert _{H}^{2},\ \ \ \ \ \ \ \forall x\in V.
\]
Here we call this condition \emph{the quasi-skew-symmetry}.
\end{assumption}

\begin{remark}\label{rem:ass}\rm
(1) Assumption \ref{ass:onAB} is a kind of \emph{coercivity condition} (cf.
\cite{krylov1981stochastic}), which ensures the solvability of SEEs of form
\eqref{eq:state}. Indeed, in view of a well-known result
(\cite{krylov1981stochastic}), SEE \eqref{eq:state} has a unique solution
$x(\cdot)\in L^{2}_{\mathbb{F}}(\Omega;C([0,1];H))\cap L^{2}(\Omega\times
[0,1],\cP,V)$ for each $u(\cdot)\in U_{\text{ad}}$ under Assumptions
\ref{ass:onAB} and \ref{ass:onfglh}.

(2) In this paper, the \emph{quasi-skew-symmetric condition} is used to
establish the $L^{p}$-estimate ($p>2$) for solutions to stochastic
evolution equations. Such a condition is refined from many concrete examples,
for instance, the nonlinear filtering equation, and other
stochastic parabolic PDEs (cf.
\cite{chow2007stochastic,rozovsky1990stochastic}). One can give other
characterizations of this condition. Indeed, for any given
$B\in\mathfrak{B}(V; H)$, the following statements are equivalent:
(i) for any $x\in V$, $\vert\langle x,Bx\rangle\vert\leq K\Vert
x\Vert_{H}^{2}$; (ii) $B+B^{\ast}\in\mathfrak{B}(H)$, where $B^{\ast}$ is the
dual operator of $B$; and (iii) there are a skew-symmetric operator
$S\in\mathfrak{B}(V; H)$ and a symmetric operator $T\in\mathfrak{B}(H)$ such
that $B=S+T$.
\end{remark}

\subsection{Main results}

The main theorem in the paper is Theorem \ref{thm:mp} --- the maximum
principle. As a preliminary, we first state the following result. Hereafter we
denote $\mathbf{E}_{t}[\,\cdot\,]=\mathbf{E} [\,\cdot\,|\mathcal{F}_{t}]$.

\begin{theorem}\label{thm:repres}
Let $A$ and $B$ satisfy Assumptions \ref{ass:onAB} and
\ref{ass:onB}, $M \in
L_{\mathrm{w}}^{2}(\Omega,\mathcal{F}_{1},\mathfrak{B}(H))$ and $N(\cdot) \in
L_{\mathrm{w}}^{2}(\Omega\times[0,1],\cP,\mathfrak{B}(H))$. In this case, we
say the four-tuple $(A,B;M,N)$ is \emph{``appropriate''}. We formally define,
for each $\tau\in\lbrack0,1)$, a stochastic bilinear function on the Banach
space $L^{4}(\Omega,\mathcal{F}_{\tau},H)$, associated with the four-tuple
$(A,B;M,N)$, as the form:
\begin{equation}\label{eq:defofT}
\lbrack T_{\tau}(A,B;M,N)](\xi_1,\xi_2):=\mathbf{E}_{\tau}\left\langle
z^{\tau,\xi_1}(1),Mz^{\tau,\xi_2}(1)\right\rangle +\mathbf{E}_{\tau}\int_{\tau
}^{1}\left\langle z^{\tau,\xi_1}(t),N(t)z^{\tau,\xi_2}(t)\right\rangle
\,\mathrm{d}t
\end{equation}
with the processes $z^{\tau,\xi_i}(\cdot)$ $(i=1,2)$ satisfying
\begin{equation}
z^{\tau,\xi_i}(t)=\xi_i+\int_{\tau}^{t}A(s)z^{\tau,\xi_i}(s)\,\mathrm{d}s+\int
_{\tau}^{t}B(s)z^{\tau,\xi_i}(s)\,\mathrm{d}W_{s},\ \ \ \ t\in\lbrack
\tau,1].\label{eq:defofz}
\end{equation}

Under the above setting, we have the following assertions:

{\rm a)} For each $\tau\in\lbrack0,1]$ and any $\xi,\zeta\in
L^{4}(\Omega,\mathcal{F}_{\tau},H)$, $[T_{\tau}(A,B;M,N)](\xi,\zeta)$ is
uniquely determined and belongs to $L^1(\Omega,\cF_{\tau},\R)$.

{\rm b)} There exists a unique 
$${P}_{\cdot}\in
L_{\mathrm{w}}^{2}(\Omega\times\lbrack0,1],\cP,\mathfrak{B}(H))
\cap
C_{\mathbb{F}}([0,1],L_{\mathrm{w}}^{2}(\Omega,\mathfrak{B}(H)))
$$ 
such that for
each $\tau\in\lbrack0,1]$ and any $\xi,\zeta\in L^{4}
(\Omega,\mathcal{F}_{\tau},H)$,
\begin{equation}
\left\langle \xi,{P}_{\tau}\zeta\right\rangle =\left[
T_{\tau}(A,B;M,N)\right] (\xi,\zeta)\ \ (\text{a.s.})\text{.}
\label{eq:relationofPT}
\end{equation}
We call ${P}_{\cdot}$ the \emph{representation} of
$T_{\cdot}(A,B;M,N)$.
\end{theorem}

The proof of this theorem is placed in Subsection 4.1.

Our main result is the following

\begin{theorem}
[Maximum Principle]\label{thm:mp}Let Assumptions \ref{ass:onAB}--\ref{ass:onB}
be satisfied. Define the Hamiltonian $\mathcal{H}:[0,1]\times H\times U\times
H\times H\to\R$ as the form
\begin{equation}\label{eq:hamiltonian}
\mathcal{H}(t,x,u,p,q):=l(t,x,u)+\left\langle p,f(t,x,u)\right\rangle
+\left\langle q,g(t,x,u)\right\rangle.
\end{equation}
Suppose $\bar{x}(\cdot)$ is the state process with
respect to an optimal control $\bar{u}(\cdot)$. Then

\emph{i)} \emph{(first-order adjoint process)} the backward stochastic
evolution equation (BSEE)
\begin{equation}\label{eq:1adjoint}
\bigg\{\begin{aligned}
\text{\emph{d}}p(t)  &  =-[A^{\ast}(t)p(t)+B^{\ast}(t)q(t)+\mathcal{H}%
_{x}(t,\bar{x}(t),\bar{u}(t),p(t),q(t))]\,\text{\emph{d}}%
t+q(t)\,\text{\emph{d}}W_{t},\\
p(1)  &  =h_{x}(\bar{x}(1))
\end{aligned}
\end{equation}

has a unique (weak) solution\,\footnote{For the definition of (weak) solutions
to BSEEs, we refer to \cite[Def. 3.1]{du2010revisit}} $(p(\cdot),q(\cdot))$;

\emph{ii)} \emph{(second-order adjoint process) }the four-tuple $(\tilde
{A},\tilde{B};\tilde{M},\tilde{N})$ with
\begin{align*}
&  \tilde{A}(t) :=A(t)+f_{x}(t,\bar{x}(t),\bar{u}(t)),~~\tilde{B}%
(t):=B(t)+g_{x} (t,\bar{x}(t),\bar{u}(t)),\\
&  \tilde{M} :=h_{xx}(\bar{x}(1)),~~\tilde{N}(t):=\mathcal{H}_{xx}(t,\bar
{x}(t),\bar{u}(t),p(t),q(t))
\end{align*}
is \emph{``appropriate"}; then by Theorem \ref{thm:repres} there exists a
unique weakly predictable operator-valued process ${P}_{\cdot}$
as the
{representation} of $T_{\cdot}(\tilde
{A},\tilde{B};\tilde{M},\tilde{N})$;

\emph{iii)} \emph{(maximum condition) }for each $u\in U$, the inequality
\begin{align}\label{eq:mpcondition}
&  \mathcal{H}(t,\bar{x}(t),u,p(t),q(t))-\mathcal{H}(t,\bar
{x}(t),\bar{u}(t),p(t),q(t))\\
&  +\frac{1}{2}\left\langle g(t,\bar{x}(t),u)-g(t,\bar{x}(t
),\bar{u}(t)),P_{t}[g(t,\bar{x}(t),u)-g(t,\bar{x}(t),\bar{u}%
(t))]\right\rangle \geq 0 \nonumber
\end{align}
holds for a.e. $(t,\omega)\in[0,1)\times\Omega$.
\end{theorem}

The proof of this theorem will be completed in Section 5.

\begin{remark}\rm
i) The inequality \eqref{eq:mpcondition} holds almost surely on the product space $[0,1)\times\Omega$, while the predictable version of process $P_{\cdot}$ insures the $(t,\omega)$-joint measurability of the left-hand side of this inequality.

ii) Let us single out a couple of important special cases: 1) The diffusion does
not contain the control variable, i.e., $g(t,x,u)\equiv g(t,x)$. In this case,
\eqref{eq:mpcondition} becomes
\[\mathcal{H}(t,\bar{x}(t),\bar{u}(t),p(t),q(t))=\min_{u\in
U}\mathcal{H}(t,\bar{x}(t),u,p(t),q(t)).\] This is a well known result, cf.
\cite{hu1990maximum}. 2) The control domain (a subset of a separable Hilbert
space $U$) is convex and all the coefficients are ${C}^1$ in $u$. Then from
\eqref{eq:mpcondition} we can deduce
\[\langle \mathcal{H}_u(t,\bar{x}(t),\bar{u}(t),p(t),q(t)), u -
\bar{u}(t)\rangle_U \ge 0,\quad \forall\,u\in U,\textrm{~a.e.~}(t,\omega).\]
This is called a local form of the maximum principle, coinciding with the
result of Bensoussan \cite{bensoussan1983stochastic}. In both cases, Assumption
\ref{ass:onfglh} can be weakened, i.e., only the first G\^ateaux
differentiability of the coefficients is in force.
\end{remark}

\section{$L^{p}$-estimates for stochastic evolution equations}

The $L^{p}$-estimate of the solutions to stochastic evolution equations plays a basic role
in our approach. Now we consider the following linear equation
\begin{equation}\label{eq:linearSEE}
\bigg\{\begin{aligned}
\mathrm{d}y(t)  &  =[\tilde{A}(t)y(t)+a(t)]\,\mathrm{d}t+[\tilde{B}(t)y(t)+b(t)]\mathrm{d}%
W_{t},\\
y(0)  &  =y_{0}\in H.
\end{aligned}
\end{equation}
Under Assumption \ref{ass:onAB}, the above equation has a unique solution
$y(\cdot)\in L^{2}_{\mathbb{F}}(\Omega,C([0,1],H))$ providing
$a(\cdot),b(\cdot)\in L^{2}(\Omega\times[0,1],\cP,H))$, see
\cite{krylov1981stochastic}. If, in addition, the operator $\tilde{B}$ satisfies
\emph{the quasi-skew-symmetric condition}, then we have

\begin{lemma}
\label{lem:Lpestimate} Let $\tilde{A}$ and $\tilde{B}$ satisfy Assumptions \ref{ass:onAB} and
\ref{ass:onB}, and $p\ge 1$. Then the solution to equation \eqref{eq:linearSEE} satisfies
\[
\mathbf{E}\sup_{t\in[0,1]}\left\Vert y(t)\right\Vert _{H}^{2p}\leq
C(\kappa,K,p)\,\E \bigg[\Vert y_{0}\Vert_{H}^{2p}+\bigg(\int_0^1 \Vert a(t)\Vert_H\d t\bigg)^{2p} + \bigg(\int_0^1 \Vert b(t)\Vert_H^2 \d t\bigg)^{p}\bigg],
\]
provided the right-hand side is finite.
\end{lemma}

\begin{proof}
In view of Assumptions
\ref{ass:onAB} and \ref{ass:onB}, we have
\begin{equation}\label{l311}
\Vert \tilde{B}y(t)\Vert _{H}^{2}\leq C(K)\Vert y(t)\Vert
_{V}^{2},\quad \langle y(t),\tilde{B}y(t)\rangle \leq K\Vert
y(t)\Vert _{H}^{2},
\end{equation}
and furthermore,
\begin{align}
& 2\langle
y(t),\tilde{A}y(t)+a(t)\rangle _{\ast}+\Vert \tilde{B}y(t)+b(t)\Vert _{H}^{2}    \label{l312}\\
=\,& 2\langle
y(t),\tilde{A}y(t)\rangle _{\ast}+\Vert \tilde{B}y(t)\Vert _{H}
^{2}+2\langle \tilde{B}y(t),b(t)\rangle +  \Vert b(t)\Vert_{H}^{2}+2\langle
y(t),a(t)\rangle \nonumber\\
\le\,& -\kappa \Vert y(t)\Vert^2_{V} + K \Vert y(t)\Vert^2_{H} + C(K)\Vert y(t)\Vert_{V}\Vert b(t)\Vert_{H} + \Vert b(t)\Vert_{H}^{2}+ 2\langle
y(t),a(t)\rangle \nonumber\\
\le\,&  K \Vert y(t)\Vert^2_{H} + C(\kappa,K)\Vert b(t)\Vert^2_{H} + 2\Vert y(t)\Vert_{H}\Vert a(t)\Vert_{H}, \nonumber\\
&\vert \langle y(t) ,\tilde{B}y(t)+b(t)\rangle \vert ^{2}
\le 2\vert \langle y(t),\tilde{B}y(t)\rangle \vert ^{2}
+ 2\vert \langle y(t),b(t)\rangle \vert ^{2}   \label{l313}\\
\le\,& 2 K^2 \Vert y(t)\Vert^4_{H} + 2\Vert y(t)\Vert^2_{H}\Vert b(t)\Vert^2_{H}. \nonumber 
\end{align}
Define stopping time $\tau_{k}:=\inf\{t\in[0,1):\Vert y(t)\Vert _{H}^{2p}>k\}
\wedge 1$. With $\lambda,\epsilon>0$, it follows from the
Burkholder-Davis-Gundy and H\"older inequalities that
\begin{align}
&\E\sup_{t\in[0,\tau_{k}]}\left\vert\int_0^{t} \e^{-\lambda s}\Vert y(s)\Vert _{H}^{2(p-1)}\langle
y(s),\tilde{B}y(s)+b(s)\rangle \,\mathrm{d}W_{s}\right\vert   \label{l31m}\\
\le ~& C\,\E\left[\int_0^{\tau_{k}} \e^{-2\lambda s}\Vert y(s)\Vert _{H}^{4(p-1)}\vert\langle
y(s),\tilde{B}y(s)+b(s)\rangle\vert^2 \,\mathrm{d}s\right]^{1/2} \nonumber\\
\le ~& C\,\E\left[\Big(\sup_{s\in[0,\tau_{k}]}\e^{-\lambda s}\Vert y(s)\Vert _{H}^{2p}\Big)\int_0^{\tau_{k}} \e^{-\lambda s}\Big(\Vert y(s)\Vert _{H}^{2p}+  \Vert y(s)\Vert _{H}
^{2p-2}\Vert b(s)\Vert _{H}^{2}\Big)\,\mathrm{d}s\right]^{1/2} \nonumber\\
\le ~& \epsilon\,\E \sup_{s\in[0,\tau_{k}]}\e^{-\lambda s}\Vert y(s)\Vert _{H}^{2p} + \frac{C}{\eps}\E \int_0^{\tau_{k}} \e^{-\lambda s}\Big(\Vert y(s)\Vert _{H}^{2p}+  \Vert y(s)\Vert _{H}
^{2p-2}\Vert b(s)\Vert _{H}^{2}\Big)\,\mathrm{d}s \nonumber\\
&\E \int_0^{\tau_{k}} \e^{-\lambda s}\Big(\Vert y(s)\Vert _{H}
^{2p-1}\Vert a(s)\Vert _{H}
+ \Vert y(s)\Vert _{H}
^{2p-2}\Vert b(s)\Vert _{H}^{2}\Big)\,\mathrm{d}s    \label{l31ab}\\
\le~ & \epsilon\,\E \sup_{s\in[0,\tau_{k}]}\e^{-\lambda s}\Vert y(s)\Vert _{H}^{2p} + \frac{C}{\eps}\E \bigg[\bigg(\int_0^{\tau_k} \e^{-\frac{\lambda s}{2p}} \Vert a(s)\Vert_H\,\d s\bigg)^{2p} + \bigg(\int_0^{\tau_k}
\e^{-\frac{\lambda s}{p}}\Vert b(s)\Vert_H^2 \,\d s\bigg)^{p}\bigg]\nonumber\\
\le~ & \epsilon\,\E \sup_{s\in[0,\tau_{k}]}\e^{-\lambda s}\Vert y(s)\Vert _{H}^{2p} + \frac{C}{\eps}\E \bigg[\bigg(\int_0^1 \Vert a(s)\Vert_H\,\d s\bigg)^{2p} + \bigg(\int_0^1 \Vert b(s)\Vert_H^2 \,\d s\bigg)^{p}\bigg]. \nonumber 
\end{align}

On the other hand, applying the It\^{o} formula (see
\cite{krylov1981stochastic}) to $\e^{-\lambda t}\Vert y(t)\Vert _{H}^{2p}$, we
have
\begin{align}
&\e^{-\lambda (t\wedge\tau_{k})}\Vert y(t\wedge\tau_{k})\Vert _{H}^{2p}-\Vert y_0 \Vert_{H}^{2p}
+ \lambda \int_0^{t\wedge\tau_{k}} \e^{-\lambda s}\Vert y(s)\Vert _{H}^{2p}\,\mathrm{d}s    \label{l31i}\\
=~& p \int_0^{t\wedge\tau_{k}} \e^{-\lambda s}\Vert y(s)\Vert _{H}^{2(p-1)}\big(2\langle
y(s),\tilde{A}y(s)+a(s)\rangle _{\ast}+\Vert \tilde{B}y(s)+b(s)\Vert _{H}^{2}\big)\,\mathrm{d}s \nonumber\\
& +2p(p-1) \int_0^{t\wedge\tau_{k}} \e^{-\lambda s}\Vert y(s)\Vert _{H}^{2(p-2)}\vert\langle
y(s),\tilde{B}y(s)+b(s)\rangle\vert^2\,\mathrm{d}s \nonumber\\
& + 2p \int_0^{t\wedge\tau_{k}} \e^{-\lambda s}\Vert y(s)\Vert _{H}^{2(p-1)}\langle
y(s),\tilde{B}y(s)+b(s)\rangle \,\mathrm{d}W_{s} \nonumber
\end{align}
Now we take $\E\sup_{t\in[0,\tau_k]}$ on both sides of the above equality. Then
by virtue of \eqref{l312}--\eqref{l31i} with sufficiently large $\lambda$ and
small $\epsilon$, we obtain
\begin{align*}
& \E \sup_{t\in[0,\tau_{k}]} \e^{-\lambda t}\Vert y(t)\Vert _{H}^{2p}
\leq C(\kappa,K,p)\,\E \bigg[\Vert y_{0}\Vert_{H}^{2p}+\bigg(\int_0^1 \Vert a(t)\Vert_H\d t\bigg)^{2p} + \bigg(\int_0^1 \Vert b(t)\Vert_H^2 \d t\bigg)^{p}\bigg].
\end{align*}
Sending $k$ to infinity, we can easily conclude the lemma.
\end{proof}

\begin{remark}{\rm
The \emph{quasi-skew-symmetric condition} is unnecessary in the case of
$p=1$, since the term $2p(p-1)\Vert y(t)\Vert_{H}^{2(p-2)}|\langle
y(t),\tilde{B}y(t)+b(t)\rangle|^{2}\,\mathrm{d}t$ does not appear in this case.
}\end{remark}

Proceeding a similar argument, we have the $L^{p}$-estimate for SEE \eqref{eq:state}.

\begin{corollary}
\label{cor:Lpforstateeq} Under Assumptions \ref{ass:onAB}--\ref{ass:onB}, the
solution $x(\cdot)$ to SEE \eqref{eq:state} satisfies
\[
\mathbf{E}\sup_{t\in\lbrack0,1]}\left\Vert x(t)\right\Vert _{H}^{p}\leq
C(\kappa,K)\sup_{t\in\lbrack0,1]} \mathbf{E}\big(1+\left\vert u(t)\right\vert
_{U}^{p}\big)
\]
with $p\in[2,4+\delta]$.
\end{corollary}

\section{Investigation into a stochastic bilinear function}

The purpose of this section is to prove Theorem \ref{thm:repres} and another
important property of the representation process ${P}_{\cdot}$.

\subsection{Proof of Theorem \ref{thm:repres}}

For convenience, we write $T_{\tau}(A,B;M,N)$ as $T_{\tau}$ to the end of this
section, and hereafter denote $C=C(\kappa,K).$

\emph{Step 1.} In view of a known result (cf. \cite{krylov1981stochastic}) , equation
\eqref{eq:defofz} has a unique solution for each $\xi\in
L^{4}(\Omega,\mathcal{F}_{\tau},H)$; moreover, it follows from Lemma \ref{lem:Lpestimate} that
\begin{equation}
\mathbf{E}_{\tau}\sup_{t\in\lbrack\tau,1]}\left\Vert z^{\tau
,\xi}(t)\right\Vert _{H}^{4}\leq C\left\Vert \xi\right\Vert _{H}^{4}.
\label{eq:estforz}
\end{equation}
Indeed, for any set $E\in \cF_{\tau}$, we obtain from Lemma \ref{lem:Lpestimate} that
\begin{equation*}
\mathbf{E}\Big({\bf 1}_{E}\cdot\sup_{t\in\lbrack\tau,1]}\left\Vert z^{\tau
,\xi}(t)\right\Vert _{H}^{4}\Big)
=\mathbf{E}\sup_{t\in\lbrack\tau,1]}\left\Vert z^{\tau
,{\bf 1}_{E}\xi}(t)\right\Vert _{H}^{4}\leq C \,\E\left\Vert {\bf 1}_{E}\xi\right\Vert _{H}^{4} = C \,\E\big({\bf 1}_{E}\cdot\left\Vert \xi\right\Vert _{H}^{4}\big),
\end{equation*}
which implies \eqref{eq:estforz}. 
In what follows, we define
\begin{equation}\label{eq:condofMN}
\Lambda:=\left\Vert M\right\Vert _{\mathfrak{B}(H)}^{2}+\int_{0}^{1}\left\Vert
N(t)\right\Vert _{\mathfrak{B}(H)}^{2}\,\mathrm{d}t\in
L^{1}(\Omega,\mathcal{F}_{1},\mathbb{R}), 
\end{equation}
and $\Lambda_t = \mathbf{E}_t[\Lambda]$. The process $(\Lambda_t)_{t\ge 0}$ is
a Doob's martingale and has a continuous version.
Then we know, fo each $\tau\in[0,1]$ and  any $\xi,\zeta\in
L^{4}(\Omega,\mathcal{F}_{\tau},H)$,
\begin{equation}\label{eq:proof:401}
\vert T_{\tau}(\xi,\zeta) \vert \leq 
C \sqrt{\Lambda_\tau} \Vert \xi \Vert_H \Vert \zeta \Vert_H \quad \textrm{(a.s.)}.
\end{equation}

Therefore, for any $\xi,\zeta\in
L^{4}(\Omega,\mathcal{F}_{\tau},H)$, $T_{\tau}(\xi,\zeta)$ is uniquely
determined and, by the H\"older inequality, belongs to
$L^{1}(\Omega,\mathcal{F}_{\tau},\R)$. The assertion (a) is proved.

\emph{Step 2. } Next we shall prove 
\begin{equation}\label{con.T}
T_{\cdot}(x,y)\in
C_{\mathbb{F}}([0,1],L^2(\Omega,\R)),\quad \forall~x,y\in H.
\end{equation}
For convenience, we denote
\begin{equation}\label{def.Y}
Y^{x,y}_{t}:=\left\langle
z^{t,x}(1),Mz^{t,y}(1)\right\rangle +\int_{t}^{1}\left\langle
z^{t,x}(r),N(r)z^{t,y}(r)\right\rangle \,\mathrm{d}r,
\quad \textrm{for }x,y\in H,
\end{equation}
with $t\in [0,1)$, and
then  $T_{t}(x,y)= \E_t [Y^{x,y}_t]$. Now we have
\begin{align*}
\lim_{s\rightarrow t}\mathbf{E}\left\vert \mathbf{E}_{s}[Y^{x,y}_{s}]-\mathbf{E}%
_{t}[Y^{x,y}_{t}]\right\vert  &  =\lim_{s\rightarrow t}\mathbf{E}\left\vert
\mathbf{E}_{s}[Y^{x,y}_{s}-Y^{x,y}_{t}]-\left(  \mathbf{E}_{t}[Y^{x,y}_{t}]-\mathbf{E}_{s}%
[Y^{x,y}_{t}]\right)  \right\vert \\
&  \leq \lim_{s\rightarrow t}\mathbf{E}\left\vert Y^{x,y}_{s}-Y^{x,y}_{t}\right\vert
+\lim_{s\rightarrow t}\mathbf{E}\left\vert \mathbf{E}_{t}[Y^{x,y}_{t}%
]-\mathbf{E}_{s}[Y^{x,y}_{t}]\right\vert.
\end{align*}
Without loss of generality, we assume $t<s$. On the one hand, the process
$(\mathbf{E}_{r}[Y^{x,y}_{t}])_{r\geq0}$ is a uniformly integrable martingale, thus
it follows from the Doob martingale convergence theorem (cf.
\cite{doob1953stochastic}) that
\[
\lim_{s\rightarrow t}\mathbf{E}\left\vert \mathbf{E}_{t}[Y^{x,y}_{t}]-\mathbf{E}%
_{s}[Y^{x,y}_{t}]\right\vert=0.
\]
On the other hand, note that
\begin{align*}
\left\vert Y^{x,y}_{s}-Y^{x,y}_{t}\right\vert \leq~  &  \left\vert \left\langle
z^{s,x}(1),M z^{s,y}(1)\right\rangle-\left\langle z^{t,x}(1),M
z^{t,y}(1)\right\rangle \right\vert \\
&  +\int_{s}^{1}\left\vert \left\langle z^{s,x}(r),N(r)z^{s,y}(r)\right\rangle
-\left\langle z^{t,x}(r),N(r)z^{t,y}(r)\right\rangle
\right\vert \mathrm{d}r\\
&  +\int_{t}^{s}\left\vert \left\langle z^{t,x}(r),N(r)z^{t,y}(r)\right\rangle
\right\vert \mathrm{d}r\\
=:\,  &  \, I_{1}+I_{2}+I_{3}.
\end{align*}
First, it follows from \eqref{eq:estforz} and the H\"older inequality that
\begin{equation}
\left\vert\,\mathbf{E} I_{3}\right\vert ^{2}\leq C\Lambda_{0}\left\vert
t-s\right\vert \left\Vert x\right\Vert _{H}^{2}\left\Vert y\right\Vert
_{H}^{2}\rightarrow0, \label{eq:proof.301}
\end{equation}
as$\ \left\vert s-t\right\vert \rightarrow0$. Next, from \eqref{eq:estforz} and
the continuity of the solution to \eqref{eq:defofz}, we have
\begin{align*}
\left\vert\,\mathbf{E} I_{2}\right\vert ^{2}~  \leq~&  C\Lambda_{0}
\Big(\left\Vert x\right\Vert _{H}^{2}\sqrt{\mathbf{E}\left\Vert y-z^{t,y}
(s)\right\Vert _{H}^{4}}\\
&  +\left\Vert y\right\Vert _{H}^{2}\sqrt{\mathbf{E}\left\Vert x-z^{t,x}
(s)\right\Vert _{H}^{4}}\,\Big)\rightarrow0,\ \ \text{as}\ \left\vert
s-t\right\vert \rightarrow0.
\end{align*}
Similarly, we can show
\[
\left\vert\,\mathbf{E} I_{1}\right\vert ^{2}\rightarrow0,\ \ \ \ \text{as\ }
\ \left\vert s-t\right\vert \rightarrow0.
\]
Therefore, we have
\begin{equation}\label{eq:L1}
\lim_{s\rightarrow t}\mathbf{E}\big\vert T_{s}(x,y)-T_{t}(x,y) \rangle\big\vert=0.
\end{equation}
This implies $T_{\cdot}(x,y)\in {C}_{\mathbb{F}}
([0,1],L^{1}(\Omega,\mathbb{R}))$. Next, it follows from \eqref{eq:proof:401}
and the Doob martingale convergence theorem that
\begin{align*}\label{eq:proof.302}
\big\vert T_{s}(x,y)-T_{t}(x,y) \big\vert^2
~\le~ & C \Vert
x\Vert_H^2\Vert y\Vert_H^2\cdot(\Lambda_s+\Lambda_t)\nonumber\\
\mathop{\longrightarrow}^{L^1}~ &C \Vert x\Vert_H^2\Vert y\Vert_H^2\cdot
2\Lambda_t\ \ \ \ \textrm{as}\ s\to t.
\end{align*}
This along with \eqref{eq:L1} and the Lebesgue dominated convergence theorem
yields
\begin{equation}\label{eq:proof.303}
\lim_{s\rightarrow t}\mathbf{E}\big\vert T_{s}(x,y)-T_{t}(x,y) \rangle\big\vert^2=0.
\end{equation}
Thus we have $T_{\cdot}(x,y) \in
C_{\mathbb{F}}([0,1],L^2(\Omega,\R))$.

\emph{Step 3.}\footnote{The argument in this step is borrowed from \cite{tessitore2013stochastic}.} 
Now we shall prove that, for any $x,y\in H$, there is a predictable modification of the process $T_{\cdot}(x,y)$.

Recalling \eqref{def.Y}, let $Y_{t}^{x,y}(\omega) = Y_{1}^{x,y}(\omega)$ when $t > 1$. Then, for any $x,y\in H$, the mapping $(t,\omega) \mapsto Y_{t}^{x,y}(\omega)$ is $\mathcal{B}([0,\infty))\times\cF$-measurable.

Let $(Y_{t}^{x,y})^{+} = Y_{t}^{x,y} \vee 0$ and $(Y_{t}^{x,y})^{-} = (-Y_{t}^{x,y}) \vee 0$. 
In view of a well-known result (c.f. \cite[Theorem 5.2]{he1992semimartingale}), there a unique predicable projection of $(Y_{t}^{x,y})^{\pm}$, denoted by $\lp{Y}^{x,y,\pm}_{\cdot}$, such that for every predictable time $\sigma$,
$$ \E[({Y}^{x,y}_{\sigma})^{\pm}1_{\{\sigma<\infty\}} \vert \cF_{\sigma-}] 
= \lp{Y}^{x,y,\pm}_{\sigma}1_{\{\sigma<\infty\}}
\quad \textrm{(a.s.).}
$$
With the continuity of filtration $\mathbb{F}$ in mind, we have for every $t\in [0,1]$,
$$ \lp{Y}^{x,y,\pm}_{t}
= 
\E[({Y}^{x,y}_{t})^{\pm}\vert \cF_{t-}]
= \E_t[({Y}^{x,y}_{t})^{\pm}]
\quad \textrm{(a.s.).}
$$
Thus we obtain 
$$ \lp{Y}^{x,y}_{t}
:=  \lp{Y}^{x,y,+}_{t} - \lp{Y}^{x,y,-}_{t}
= \E_t[{Y}^{x,y}_{t}]
= T_t(x,y)
\quad \textrm{(a.s.),}
$$
which implies $\lp{Y}^{x,y}_{\cdot}$ is a predictable version of $T_{\cdot}(x,y)$.

\emph{Step 4.} The construction of process $P_{\cdot}$.

Take a standard
complete orthonormal basis $\{e_{i}\}$ in $H$, and a predictable version of $T_{\cdot}(e_i,e_j)$ for each $i,j\in\mathbb{N}$. Set
\[
\Gamma_{ij}=\{(t,\omega)\in\lbrack0,1] \times\Omega:T_{t}(e_i,e_j)(\omega)\leq
C \sqrt{\Lambda_{t}(\omega)}\},
\]
where the constant $C$ is taken from \eqref{eq:proof:401}. Then $\Gamma_{ij}$
is a predictable set with full measure, and so
$\Gamma:=\cap_{i,j\in\mathbb{N}}\,\Gamma_{ij}$ is also a predictable set with
full measure; moreover, in view of \eqref{eq:proof:401}, the section $\Gamma(t) := \{\omega:(t,\omega)\in \Gamma\}$ is a set of probability $1$ for each $t\in [0,1]$.

Thanks to the Riesz representation theorem, there is a unique
$P_{t}(\omega)\in \mathfrak{B}(H)$ for each $(t,\omega)\in\Gamma$ such that%
\[
\left\langle e_{i},P_{t}(\omega)e_{j}\right\rangle _{H} = T_{t}(e_i,e_j)(\omega)
\]
and $\Vert P_{t}(\omega)\Vert_{\mathfrak{B}(H)}\leq C
\sqrt{\Lambda_{t}(\omega)}$. Let $P_{t}(\omega)=0$ for
$(t,\omega)\in\Gamma^{c}$. Then $\langle e_{i} ,P_{\cdot}e_{j}\rangle$ belongs
to $L^{2}(\Omega \times\lbrack0,1],\mathcal{P},\mathbb{R})$. Since for
each $(t,\omega)$ and $x,y\in H$,
\[
\left\langle x,P_{t}(\omega)y\right\rangle = \lim_{n\to \infty} \left\langle
x_{n},P_{t}(\omega)y_{n}\right\rangle
\]
with $x_{n}=\sum_{i=1}^{n}\langle x,e_{i}\rangle e_{i}$ and $y_{n}%
=\sum_{i=1}^{n}\langle y,e_{i}\rangle e_{i}$,
we have $\langle
x,P_{\cdot}y\rangle\in L^{2}(\Omega\times\lbrack
0,1],\mathcal{P},\mathbb{R})$; moreover, from the bilinearity of $T_t(\cdot,\cdot)$, we know that $\langle
x,P_{t}y\rangle = T_t(x,y)$ a.s. for each $\in [0,1]$ and any $x,y\in H$. Recalling \eqref{con.T}, we have
$${P}_{\cdot}\in
L_{\mathrm{w}}^{2}(\Omega\times\lbrack0,1],\cP,\mathfrak{B}(H))
\cap
C_{\mathbb{F}}([0,1],L_{\mathrm{w}}^{2}(\Omega,\mathfrak{B}(H))).
$$ 

It remains to show the relation
\eqref{eq:relationofPT}. Fix arbitrary $t\in [0,1]$. It
follows from the definition of $T_{t}(\cdot,\cdot)$ that (i) for any $E\in\mathcal{F}_{t}$,
\[
T_{t}(x\mathbf{1}_{E},y\mathbf{1}_{E})=\mathbf{1}_{E} T_{t }(x,y)
\quad\textrm{(a.s.)}\quad
\forall x,y\in H;
\]
and (ii) for any
$E_{1},E_{2}\in\mathcal{F}_{t} $ with $E_{1}\cap E_{2}=\emptyset$,
\[
T_{t}(x\mathbf{1}_{E_{1}},y\mathbf{1}_{E_{2}})=0
\quad\textrm{(a.s.)}\quad
\forall x,y\in H.
\]
This means that, for any simple $H$-valued $\mathcal{F}_{t}$-measurable
random variables $\xi$ and $\zeta$, we have
\[
\langle\xi,{P}_{t}\zeta\rangle= T_{t}%
(\xi,\zeta)\quad\textrm{(a.s.)},
\]
which along with a standard argument of approximation yields the relation
\eqref{eq:relationofPT}.

The uniqueness of the representation is obvious.
The proof of Theorem \ref{thm:repres} is complete.

\subsection{An important property of ${P}_{\cdot}$}

Give the same four-tuple $(A,B;M,N)$ as in Theorem \ref{thm:repres}. Let
$\varepsilon\in(0,1-\tau)$. For each $\xi\in L^{4}(\Omega,\mathcal{F}_{\tau
},H)$, consider the equation
\begin{equation}
z_{\varepsilon}^{\tau,\xi}(t)=\int_{\tau}^{t}Az_{\varepsilon}^{\tau,\xi
}(s)\mathrm{d}s+\int_{\tau}^{t}[Bz_{\varepsilon}^{\tau,\xi}(s)
+\varepsilon^{-\frac{1}{2}}\mathbf{1}_{[\tau,\tau+\varepsilon)}\xi]\,\mathrm{d}
W_{s},\quad t\in[\tau,1] \label{eq:defofzeps}
\end{equation}
and the following bilinear functional on $L^{4}(\Omega,\mathcal{F}_{\tau},H)$
with parameter $\varepsilon$\,:
\[
\lbrack T_{\tau}^{\varepsilon}(A,B;M,N)](\xi,\zeta):=\mathbf{E}\left\langle
z_{\varepsilon}^{\tau,\xi}(1),Mz_{\varepsilon}^{\tau,\zeta}(1)\right\rangle
+\mathbf{E}\int_{\tau}^{1}\left\langle z_{\varepsilon}^{\tau,\xi
}(t),N(t)z_{\varepsilon}^{\tau,\zeta}(t)\right\rangle \,\mathrm{d}t.
\]
In view of Lemma \ref{lem:Lpestimate}, equation \eqref{eq:defofzeps} has a
unique solution $z_{\varepsilon}^{\tau,\xi}(\cdot)$ such that
\[
\mathbf{E}\sup_{t\in\lbrack0,1]}\left\Vert z_{\varepsilon}^{\tau,\xi
}(t)\right\Vert _{H}^{4}\leq C\mathbf{E}\left\Vert \xi\right\Vert _{H}^{4}.
\]
Thus $T_{\tau}^{\varepsilon}=T_{\tau}^{\varepsilon}(A,B;M,N)$ is well-defined.

Next, we shall prove a result concerning the relation between
$T_{\tau}^{\varepsilon}$ and ${P}_{\tau}$, which plays a key role in the
proof of the maximum principle.

\begin{proposition}
\label{prop:propofP} Under the above setting, we have
\begin{equation}\label{eq:pp}
\mathbf{E}\langle \xi,{P}_{\tau}\zeta\rangle =\lim_{\varepsilon
\to 0}[T_{\tau}^{\varepsilon}(A,B;M,N)](\xi,\zeta)
\end{equation}
for each $\tau\in
\lbrack0,1)$ and any $\xi,\zeta\in L^{4}(\Omega,\mathcal{F}_{\tau},H)$.
\end{proposition}

\begin{proof}
First of all, we claim: the assertion holds true if it does in a dense subset
$D_\tau$ of $L^{4}(\Omega,\mathcal{F}_{\tau},H)$, i.e., the relation
\eqref{eq:pp} holds for any $\xi,\zeta\in D_\tau$.

Indeed, from the density, for arbitrary $\eta>0$, we
can find $\xi^{\eta},\zeta^{\eta}\in D_{\tau}$ such that
\[
\mathbf{E}\Vert\xi-\xi^{\eta}\Vert_{H}^{4}+\mathbf{E}\Vert\zeta-\zeta^{\eta}%
\Vert_{H}^{4}<\eta^{4}.
\]
Bearing in mind \eqref{eq:proof:401}, one can show
\[
\left\vert \mathbf{E}\left\langle \xi,{P}_{\tau}\zeta\right\rangle
-\mathbf{E}\left\langle \xi^{\eta},{P}_{\tau}\zeta^{\eta}\right\rangle \right\vert
+\left\vert T_{\tau}^{\varepsilon}(\xi,\zeta)-T_{\tau}^{\varepsilon}(\xi
^{\eta},\zeta^{\eta})\right\vert <C\eta\sqrt{\Lambda_{0}}\big[\mathbf{E}%
(\Vert\xi\Vert_{H}^{4}+\Vert\zeta\Vert_{H}^{4})\big]^{\frac{1}{4}}.
\]
So, if the assertion holds in $D_\tau$, i.e., $\mathbf{E}\langle \xi^{\eta},{P}_{\tau}\zeta
^{\eta}\rangle =\lim_{\varepsilon\to 0}T_{\tau}^{\varepsilon}(\xi
^{\eta},\zeta^{\eta})$,
then
\[
\limsup_{\varepsilon\to  0} \left\vert \mathbf{E}\left\langle
\xi,{P}_{\tau}\zeta\right\rangle -T_{\tau
}^{\varepsilon}(\xi,\zeta)\right\vert \le C\eta\sqrt{\Lambda_{0}}\big[\mathbf{E}%
(\Vert\xi\Vert_{H}^{4}+\Vert\zeta\Vert_{H}^{4})\big]^{\frac{1}{4}}.
\]
From the arbitrariness of $\eta$, we prove the claim.

Now we define
\[
D_{\tau}:=\left\{  \xi\in \mathcal{F}_{\tau
}:\xi\text{ is a simple random variable with values in } V\right\}  .
\]
Obviously, $D_{\tau}$ is dense in $L^{4}(\Omega
,\mathcal{F}_{\tau},H)$.

The next result is the key-point in the proof, which gives a simple asymptotic
alternative of $z_{\varepsilon}^{\tau,\xi}(\tau+\varepsilon)$, independent of
the operators $A$ and $B$.
\begin{lemma}
\label{lem:appofz} Define $\xi_t=\varepsilon^{-\frac{1}{2}}(W_{t}-W_{\tau})\xi$ for $\xi\in
D_{\tau}$. Then
\[
\mathbf{E}\Vert
z_{\varepsilon}^{\tau,\xi}(\tau+\varepsilon)-\xi_{\tau+\varepsilon} \Vert
_{H}^{4}\leq C\varepsilon^{2}\,\mathbf{E}\Vert \xi\Vert_{V}^{4}.
\]
\end{lemma}
\begin{proof}
Denoting $y_{t}=z_{\varepsilon}^{\tau,\xi}(t)-\xi_t$, we can write down the equation
\begin{align*}
\mathrm{d}y_t = \big( A y_t+A\xi_t\big) \,\mathrm{d}t+\big( By_t+B\xi_t\big) \,\mathrm{d}W_{t},\quad y_{\tau}=0
\end{align*}
with $t\in\lbrack\tau,\tau+\varepsilon]$. Inspired by the proof of Lemma
\ref{lem:Lpestimate}, we can deduce
\begin{align*}
\mathbf{E}\left\Vert y_t\right\Vert_{H}^{4} ~\leq~& 2\,\E\int_{\tau}^{t}\left\Vert y_{s}\right\Vert_H^2 \left[
-\kappa\left\Vert y_s\right\Vert_V^2+C(K)\left\Vert y_s\right\Vert_H^2 + \langle y_s,A\xi_s\rangle_{*}
+\Vert B\xi_s\Vert_H^2\right]\d s\\
\leq~& C(\kappa,K)\,\E\int_{\tau}^{t}\left\Vert y_{s}\right\Vert_H^2 \left[\left\Vert y_s\right\Vert_H^2 + \Vert A\xi_s\Vert_{V^*}^2
+\Vert B\xi_s\Vert_H^2\right]\d s
\end{align*}
with $t\in\lbrack\tau,\tau+\varepsilon]$. Note that $\Vert A\xi_s\Vert_{V^*}
+\Vert B\xi_s\Vert_H\le C(K)\Vert \xi_s\Vert_V$. Then by means of the Fubini
theorem, the Gronwall and Young inequalities, we have
\begin{align*}
\sup_{t\in\lbrack\tau,\tau+\varepsilon]}\mathbf{E}\left\Vert y_t\right\Vert_{H}^{4} \,&\leq C\bigg(\int_{\tau}^{\tau+\varepsilon}\E\big(\Vert A\xi_s\Vert_{V^*}^2
+\Vert B\xi_s\Vert_H^2\big)\,\d s\bigg)^2\\
&\leq C\bigg(\int_{\tau}^{\tau+\varepsilon}\E\Vert \xi_s\Vert_{V}^2\,\d s\bigg)^2
= C \left(\frac{\eps}{2}\,\E\left\Vert \xi\right\Vert_{V}^{2}\right)^2.
\end{align*}
This concludes the lemma.
\end{proof}

Let us move on the proof of Proposition \ref{prop:propofP}. For $\xi,\zeta\in D_{\tau}$, define
\[ \xi_t:=\varepsilon^{-\frac{1}{2}}(W_{t}-W_{\tau})\xi \ \
\text{and}\ \ \zeta_t:=\varepsilon^{-\frac{1}{2}}(W_{t}-W_{\tau })\zeta.
\]
Notice the fact that for any $\xi,\zeta\in D_{\tau}$,
\begin{align*}
T_{\tau}^{\varepsilon}(\xi,\zeta)  &  =\mathbf{E}\int_{\tau}^{\tau
+\varepsilon}\left\langle z_{\varepsilon}^{\tau,\xi}(t),N(t)z_{\varepsilon
}^{\tau,\zeta}(t)\right\rangle \,\mathrm{d}t+\mathbf{E}\left\langle
z_{\varepsilon}^{\tau,\xi}(\tau+\varepsilon),{P}_{\tau+\varepsilon
}z_{\varepsilon}^{\tau,\zeta}(\tau+\varepsilon)\right\rangle \\
&  =:I_{1}+I_{2}.
\end{align*}
Now we let $\varepsilon$ tend to $0$. First, one can show, just like
\eqref{eq:proof.301}, the term $I_{1}$ tends to $0$; Next, by means of Lemma
\ref{lem:appofz} and the relation \eqref{eq:proof:401}, we have
\begin{align*}
\mathbf{E}\vert\langle\xi_{\tau+\varepsilon},{P}_{\tau+\varepsilon}
\zeta_{\tau+\varepsilon}\rangle-I_2\vert
~\leq~& C\sqrt{\mathbf{E}\Lambda_{1}}\big( \mathbf{E}\Vert\xi\Vert_{H}
^{4}\big) ^{\frac{1}{4}}\big( \mathbf{E}\Vert z_{\varepsilon}^{\tau,\zeta
}(\tau+\varepsilon)-\zeta_{\tau+\varepsilon}\Vert_{H}^{4}\big) ^{\frac{1}{4}
}\\
&  +C\sqrt{\mathbf{E}\Lambda_{1}}\big( \mathbf{E}\Vert\zeta\Vert_{H}
^{4}\big) ^{\frac{1}{4}}\big( \mathbf{E}\Vert z_{\varepsilon}^{\tau,\xi}
(\tau+\varepsilon)-\xi_{\tau+\varepsilon}\Vert_{H}^{4}\big) ^{\frac{1}{4}}\\
\rightarrow~&  0,\ \ \ \ \text{as\ }\ \varepsilon\to 0.
\end{align*}
Thus we obtain
\begin{equation}\label{lem:approfTeps}
\lim_{\varepsilon\to 0}\big\vert \mathbf{E}\left\langle
\xi_{\tau+\varepsilon},{P}_{\tau+\varepsilon
}\zeta_{\tau+\varepsilon}\right\rangle -T_{\tau}^{\varepsilon}
(\xi,\zeta)\big\vert =0,\quad \forall~\xi,\zeta\in
D_{\tau}.
\end{equation}
On the other hand, for $\xi,\zeta\in D_{\tau}$, we deduce
\begin{align*}
& \big\vert \mathbf{E}\left\langle
\xi_{\tau+\varepsilon},({P}_{\tau+\varepsilon
}-{P}_{\tau})\zeta_{\tau+\varepsilon}\right\rangle \big\vert^2
~=~   \big\vert\mathbf{E}\big[\varepsilon^{-1}\left\vert W_{\tau+\varepsilon
}-W_{\tau}\right\vert ^{2}\left\langle \xi,\left(  {P}_{\tau+\varepsilon
}-{P}_{\tau}\right)  \zeta\right\rangle \big]\big\vert^{2}\\
&\leq~  \big[\mathbf{E}\big(\varepsilon^{-2}\left\vert
W_{\tau+\varepsilon}-W_{\tau }\right\vert
^{4}\big)\big]\cdot\big[\mathbf{E}\left\vert \left\langle \xi,\left(
{P}_{\tau+\varepsilon}-{P}_{\tau}\right)  \zeta\right\rangle \right\vert ^{2}\big]
~\leq~  3\,\mathbf{E}\left\vert \left\langle \xi,\left(
{P}_{\tau+\varepsilon }-{P}_{\tau}\right)  \zeta\right\rangle \right\vert
^{2};
\end{align*}
since $\xi$ and $\zeta$ are simple random variables, recalling
\eqref{eq:proof.303}, we know that the last term in the above relation tends to
$0$ when $\eps\to 0$. Note that $\mathbf{E}\langle
\xi_{\tau+\varepsilon},{P}_{\tau} \zeta_{\tau+\varepsilon}\rangle = \mathbf{E}\langle
\xi,{P}_{\tau} \zeta\rangle $. Hence, we get
\[
\lim_{\varepsilon\to 0}\mathbf{E}\left\langle
\xi_{\tau+\varepsilon},{P}_{\tau+\varepsilon
}\zeta_{\tau+\varepsilon}\right\rangle = \mathbf{E}\left\langle
\xi,{P}_{\tau} \zeta\right\rangle,\quad \forall~\xi,\zeta\in
D_{\tau}.
\]
This along with \eqref{lem:approfTeps} yields
\[
\lim_{\varepsilon\to 0}T_{\tau}^{\varepsilon}
(\xi,\zeta) = \mathbf{E}\left\langle
\xi,{P}_{\tau} \zeta\right\rangle,\quad \forall~\xi,\zeta\in
D_{\tau},
\]
which completes the proof of Proposition \ref{prop:propofP}.
\end{proof}

\section{Proof of Theorem \ref{thm:mp}}

In this section, we shall prove the maximum principle.

\subsection{Second-order expansion of spike variation}

Assume $\bar{x}(\cdot)$ is the state process with respect to an optimal control
$\bar{u}(\cdot)$. We fix a $\tau\in [0,1)$ in this subsection.

Following a classical technique in the optimal control problem, we construct a
perturbed admissible control in the following way (named \emph{spike
variation})
\[
u^{\varepsilon}(t):=\bigg\{
\begin{aligned}
& u(t), && \text{if }t\in\lbrack\tau,\tau+\varepsilon],\\
&\bar{u}(t), && \text{otherwise,}
\end{aligned}
\]
with fixed $\tau\in\lbrack0,1)$, sufficiently small positive $\varepsilon$,
and any given admissible control $u(\cdot)$.

Let $x^{\varepsilon}(\cdot)$ be the state process with respect to the control
$u^{\varepsilon}(\cdot)$. For the sake of convenience, we denote for $\varphi
=f,g,l,f_{x},g_{x},l_{x},f_{xx},g_{xx},l_{xx},$
\begin{align*}
\bar{\varphi}(t)  &  :=\varphi(t,\bar{x}(t),\bar{u}(t)),\\
\varphi^{\Delta}(t)  &  :=\varphi(t,\bar{x}(t),u(t))-\bar
{\varphi}(t),\\
\varphi^{\Delta,\eps}(t)  &  :=\varphi^{\Delta}(t)\cdot \mathbf{1}_{[\tau,\tau+\eps]}(t),\\
\tilde{\varphi}_{xx}^{\varepsilon}(t)  &  := 2 \int_{0}^{1}\lambda\varphi_{xx}\left(
t,\lambda\bar{x}(t)+(1-\lambda)x^{\varepsilon}(t),u^{\varepsilon
}(t)\right)  \mathrm{d}\lambda.
\end{align*}

By means of the basic $L^{p}$ estimate, we have

\begin{lemma}
\label{lem:errorestimate} Under Assumptions \ref{ass:onAB}--\ref{ass:onB}, we
have
\[
\mathbf{E}\sup_{t\in\lbrack0,1]}\left\Vert \Xi^{\eps}(t)\right\Vert _{H}^{2}
:=\mathbf{E}\sup_{t\in\lbrack0,1]}\left\Vert x^{\varepsilon}(t)-\bar
{x}(t)-x_{1}^{\eps}(t)-x_{2}^{\eps}(t)\right\Vert _{H}^{2}=o(\varepsilon^{2}),
\]
where $x_{1}^{\eps}(\cdot)$ and $x_{2}^{\eps}(\cdot)$ are the solutions respectively to%
\begin{align}
x_{1}^{\eps}(t)=  & \int_{0}^{t}[A(s)x_{1}^{\eps}(s)
+\bar{f}_{x}(s)x_{1}^{\eps}(s)]\,\mathrm{d}s   \label{eq:1variation}\\
&
+\int_{0}^{t}[B(s)x_{1}^{\eps}(s)+\bar{g}_{x}(s)x_{1}^{\eps}(s)
+g^{\Delta,\eps}(s)]\mathrm{d}
W_{s},   \nonumber\\
x_{2}^{\eps}(t)=  &
\int_{0}^{t}[A(s)x_{2}^{\eps}(s)+\bar{f}_{x}(s)x_{2}^{\eps}(s)+\frac{1}{2}
\bar{f}_{xx}(s)\left(  x_{1}^{\eps}\otimes x_{1}^{\eps}\right)
(s)+f^{\Delta,\eps
}(s)]\,\mathrm{d}s    \label{eq:2variation}\\
&
+\int_{0}^{t}[B(s)x_{2}^{\eps}(s)+\bar{g}_{x}(s)x_{2}^{\eps}(s)+\frac{1}{2}\bar{g}
_{xx}(s)\left(  x_{1}^{\eps}\otimes x_{1}^{\eps}\right)
(s)+g_{x}^{\Delta,\eps}(s)x_{1}^{\eps}
(s)]\,\mathrm{d}W_{s}.    \nonumber
\end{align}
\end{lemma}

\begin{proof}
The proof is rather standard (cf. \cite{yong1999stochastic}). The
$L^p$-estimate for SEEs plays a key role here. Indeed, by means of Lemma
\ref{lem:Lpestimate}, and keeping in mind Assumption \ref{ass:onfglh} and
Corollary \ref{cor:Lpforstateeq}, we deduce that
\begin{align}
&\E\sup_{t\in[0,1]} \left\Vert x_{1}^{\eps}(t)\right\Vert _{H}^{2p} ~\leq~ C\,\E \left[ \int_{\tau}^{\tau+\eps}\!\left\Vert g^{\Delta}(t) \right\Vert_H^{2}\cdot{\bf 1}_{[\tau,\tau+\eps]}(t)\,\d t \right]^p  \label{eq:inproof_301}\\
&\leq~ C \eps^{p-1}\E \int_{\tau}^{\tau+\eps}\!\left\Vert g^{\Delta}(t) \right\Vert_H^{2p}\d t ~\leq~ C \eps^{p} \sup_{t\in[0,1]}\E \left\Vert g^{\Delta}(t) \right\Vert_H^{2p} \nonumber\\
& \le~ C \eps^{p} \sup_{t\in[0,1]}\E \left(1+\big\vert u(t)\big\vert_U^{2p} + \big\vert \bar{u}(t)\big\vert_U^{2p}\right)  
\nonumber\end{align}
with $p\in[1,2+\frac{1}{2}(\delta_{u} \wedge \delta_{\bar{u}})]$; moreover, by
similar arguments we have the following estimates:
\begin{align}
\mathbf{E}\Big\{ \varepsilon^{-p}\sup_{t}\left\Vert x^{\varepsilon}(t)-\bar
{x}(t)\right\Vert _{H}^{2p} + \varepsilon^{-2}\sup_{t}\left\Vert
x_{2}^{\eps}(t)\right\Vert _{H}
^{2} + \varepsilon^{-2}\sup_{t}\left\Vert x^{\varepsilon}(t)-\bar
{x}(t)-x_{1}^{\eps}(t)\right\Vert _{H}^{2}\Big\}\leq C. \label{eq:inproof_302}
\end{align}
On the other hand, a direct calculation gives%
\begin{align}
\Xi^{\eps}(t)=  &  \int_{0}^{t}\left[
A(s)\Xi^{\eps}(s)+\bar{f}_{x}(s)\Xi^{\eps}(s)+\alpha
^{\varepsilon}(s)\right]  \,\mathrm{d}s \label{proofxi}\\
&  +\int_{0}^{t}\left[
B(s)\Xi^{\eps}(s)+\bar{g}_{x}(s)\Xi^{\eps}(s)+\beta^{\varepsilon }(s)\right]
\,\mathrm{d}W_{s}, \nonumber
\end{align}
where%
\begin{align*}
\alpha^{\varepsilon}(s):=&\,f_{x}^{\Delta,\eps}(s)(x^{\varepsilon}(s)-\bar
{x}(s))+
\frac{1}{2}(\tilde{f}^{\eps}_{xx}(s)-\bar{f}_{xx}(s))(x^{\varepsilon}(s)-\bar
{x}(s))\otimes(x^{\varepsilon}(s)-\bar{x}(s))\\
&  +\frac{1}{2}\bar{f}_{xx}(s)\left[(x^{\varepsilon}(s)-\bar{x}%
(s))\otimes(x^{\varepsilon}(s)-\bar{x}(s))-(x_{1}^{\eps}\otimes
x_{1}^{\eps})(s)\right], \\
\beta^{\varepsilon}(s):= & \, g_{x}^{\Delta,\eps}(s)(x^{\varepsilon}(s)-\bar
{x}(s)-x_{1}^{\eps}(s))+\frac{1}{2}(\tilde{g}^{\eps}_{xx}(s)-\bar{g}_{xx}(s))(x^{\varepsilon}(s)-\bar
{x}(s))\otimes(x^{\varepsilon}(s)-\bar{x}(s))\\
&  +\frac{1}{2}\bar{g}_{xx}(s)\left[  (x^{\varepsilon}(s)-\bar{x}
(s))\otimes(x^{\varepsilon}(s)-\bar{x}(s))-(x_{1}^{\eps}\otimes
x_{1}^{\eps})(s)\right]  .
\end{align*}
Now apply Lemma \ref{lem:Lpestimate} to \eqref{proofxi}. Keeping in mind
\eqref{eq:inproof_301} and \eqref{eq:inproof_302}, and by means of the H\"older
inequality and the Lebesgue dominated convergence theorem, we conclude
\[
\mathbf{E}\sup_{t\in\lbrack0,1]}\left\Vert \Xi^{\eps}(t)\right\Vert _{H}^{2}%
\leq \mathbf{E}\bigg[\int_{0}^{1}\left\Vert \alpha^{\varepsilon
}(s)\right\Vert _{H}\,\mathrm{d}s\bigg]^{2}%
+\mathbf{E}\int_{0}^{1}\left\Vert \beta^{\varepsilon}(s)\right\Vert _{H}%
^{2}\,\mathrm{d}s = o(\varepsilon^{2}).
\]
The lemma is proved.
\end{proof}

With the aid of the above lemma and by the fact
\[
J(u^{\varepsilon}(\cdot))-J(\bar{u}(\cdot))\geq0,
\]
we can prove the following result.

\begin{lemma}
\label{lem:2expansion} Under Assumptions \ref{ass:onAB}--\ref{ass:onB}, we
have%
\begin{align}
\label{eq:variationfor}o(\varepsilon)~\leq~  &  \mathbf{E}\int_{0}%
^{1}\Big[ l^{\Delta,\eps}(t)+\left\langle \bar{l}_{x}(t),x_{1}^{\eps}(t)+x_{2}^{\eps}%
(t)\right\rangle +\frac{1}{2}\left\langle x_{1}^{\eps}(t),\bar{l}_{xx}(t)x_{1}^{\eps}%
(t)\right\rangle \Big] \mathrm{d}t \\
&  +\mathbf{E}\Big[ \left\langle h_{x}(\bar{x}(1)),x_{1}^{\eps}(1)+x_{2}^{\eps}%
(1)\right\rangle +\frac{1}{2}\left\langle x_{1}^{\eps}(1),h_{xx}(\bar{x}%
(1))x_{1}^{\eps}(1)\right\rangle \Big] .  \nonumber
\end{align}

\end{lemma}

\begin{proof}
The proof is also standard (cf. \cite{yong1999stochastic}), we give a
sketch here. A direct calculation shows that
\begin{align*}
0~\leq~  &  J(u^{\varepsilon}(\cdot))-J(\bar{u}(\cdot))\\
=~  &  \mathbf{E}\int_{0}^{1}\Big[l^{\Delta,\eps}(t)+\left\langle \bar{l}%
_{x}(t),x_{1}^{\eps}(1)+x_{2}^{\eps}(1)\right\rangle +\frac{1}{2}\left\langle x_{1}^{\eps}
(t),\bar{l}_{xx}(t)x_{1}^{\eps}(t)\right\rangle \Big]\,\mathrm{d}t\\
&  +\mathbf{E}\Big[\left\langle h_{x}(\bar{x}(1)),x_{1}^{\eps}(1)+x_{2}^{\eps}%
(1)\right\rangle +\frac{1}{2}\left\langle x_{1}^{\eps}(1),h_{xx}(\bar{x}%
(1))x_{1}^{\eps}(1)\right\rangle \Big]+\gamma(\varepsilon),
\end{align*}
where%
\begin{align*}
\gamma(\varepsilon) ~:=~  &  \mathbf{E}\left\langle h_{x}(\bar{x}%
(1)),\Xi^{\eps}(1)\right\rangle +\mathbf{E}\int_{0}^{1}\left\langle \bar{l}%
_{x}(t),\Xi^{\eps}(t)\right\rangle \,\mathrm{d}t\\
&  +\frac{1}{2}\mathbf{E}\left\langle \left[  \tilde{h}_{xx}^{\varepsilon
}-h_{xx}\left(  \bar{x}(1)\right)  \right]  (x^{\varepsilon}(1)-\bar
{x}(1)),x^{\varepsilon}(1)-\bar{x}(1)\right\rangle \\
&  +\frac{1}{2}\mathbf{E}\left\langle h_{xx}\left(  \bar{x}(1)\right)
(x^{\varepsilon}(1)-\bar{x}(1)),x^{\varepsilon}(1)-\bar{x}(1)-x_{1}^{\eps}%
(1)\right\rangle \\
&  +\frac{1}{2}\mathbf{E}\left\langle h_{xx}\left(  \bar{x}(1)\right)
(x^{\varepsilon}(1)-\bar{x}(1)-x_{1}^{\eps}(1)),x_{1}^{\eps}(1)\right\rangle\\
&  + \E\int_0^1\left\langle l^{\Delta,\eps}_x(t),x^{\eps}(t)-\bar{x}(t)\right\rangle\d t\\
&  +\frac{1}{2}\mathbf{E}\int_{0}^{1}\left\langle \left[  \tilde{l}%
_{xx}^{\varepsilon}(t)-\bar{l}_{xx}(t)\right]  (x^{\varepsilon}(t)-\bar
{x}(t)),x^{\varepsilon}(t)-\bar{x}(t)\right\rangle \,\mathrm{d}t\\
&  +\frac{1}{2}\mathbf{E}\int_{0}^{1}\left\langle \bar{l}_{xx}%
(t)(x^{\varepsilon}(t)-\bar{x}(t)),x^{\varepsilon}(t)-\bar{x}(t)-x_{1}^{\eps}%
(t)\right\rangle \,\mathrm{d}t\\
&  +\frac{1}{2}\mathbf{E}\int_{0}^{1}\left\langle \bar{l}_{xx}%
(t)(x^{\varepsilon}(t)-\bar{x}(t)-x_{1}^{\eps}(t)),x_{1}^{\eps}(t)\right\rangle
\,\mathrm{d}t
\end{align*}
with
\[
\tilde{h}_{xx}^{\varepsilon}:= 2 \int_{0}^{1}\lambda h_{xx}\left(\lambda \bar{x}(1)+(1-\lambda)
x^{\varepsilon}(1)\right)  \mathrm{d}\lambda.
\]

Consequently, by virtue of \eqref{eq:inproof_301}, \eqref{eq:inproof_302} and
the Lebesgue dominated convergence theorem, we can deduce $\vert
\gamma(\varepsilon)\vert = o(\varepsilon)$, which implies the lemma.
\end{proof}

\subsection{First-order duality analysis}

We need do some duality analysis in order to get the maximum condition
\eqref{eq:mpcondition} by sending $\varepsilon$ to $0$ in
inequality \eqref{eq:variationfor}. In this subsection, we still fix the $\tau\in[0,1)$.
Recall the Hamiltonian
\[
\mathcal{H}(t,x,u,p,q)=l(t,x,u)+\left\langle p,f(t,x,u)\right\rangle
+\left\langle q,g(t,x,u)\right\rangle ,
\]
and BSEE \eqref{eq:1adjoint}. Under Assumptions \ref{ass:onAB} and
\ref{ass:onfglh}, it follows from Du-Meng \cite[Propostion 3.2]{du2010revisit}
that equation \eqref{eq:1adjoint} has a unique {weak solution} $\left(
p(\cdot),q(\cdot)\right)  $ with the estimate
\begin{equation}
\mathbf{E}\sup_{t\in[0,1]}\left\Vert p(t)\right\Vert _{H}^{2}%
+\mathbf{E}\int_{0}^{1}\left\Vert q(t)\right\Vert _{H}^{2}\,\mathrm{d}t\leq
C(\kappa,K)\sup_{t\in[0,1]}\E \big(1+\left\vert \bar
{u}(t)\right\vert _{U}^{2}\big). \label{eq:estforajoint}%
\end{equation}
Thus the assertion (i) of Theorem \ref{thm:mp} holds true. Furthermore, from
Lemma \ref{lem:2expansion} we have

\begin{lemma}
\label{cor:2expansion} Under Assumptions \ref{ass:onAB}--\ref{ass:onB}, we
have%
\begin{align}
o(1)~\leq~  &  \varepsilon^{-1}\mathbf{E}\int_{\tau}^{\tau+\varepsilon}\left[
\mathcal{H}(t,\bar{x}(t),u(t),p(t),q(t))-\mathcal{H}(t,\bar{x}(t),\bar
{u}(t),p(t),q(t))\right]  \,\mathrm{d}t \label{eq:varitionineq}\\
&  +\frac{1}{2}\varepsilon^{-1}\mathbf{E}\int_{0}^{1}\left\langle
x_{1}^{\eps}(t),\mathcal{H}_{xx}(t,\bar{x}(t),\bar{u}(t),p(t),q(t))x_{1}^{\eps}%
(t)\right\rangle \,\mathrm{d}t\nonumber\\
&  +\frac{1}{2}\varepsilon^{-1}\mathbf{E}\left\langle
x_{1}^{\eps}(1),h_{xx}(\bar {x}(1))x_{1}^{\eps}(1)\right\rangle ,  \nonumber
\end{align}
where $(p(\cdot),q(\cdot))$ is the solution to equation \eqref{eq:1adjoint}.
\end{lemma}

\begin{proof}
In view of the dual relation between the SEE and BSEE (or the It\^o formula), and
by \eqref{eq:inproof_301} and \eqref{eq:estforajoint}, we have
\begin{align*}
\mathbf{E}\int_{0}^{1}  &  \left[  \left\langle \bar{l}_{x}(t),x_{1}^{\eps}
(t)+x_{2}^{\eps}(t)\right\rangle \right]  \,\mathrm{d}t+\mathbf{E}\left\langle
h_{x}(\bar{x}(1)),x_{1}^{\eps}(1)+x_{2}^{\eps}(1)\right\rangle \\
=~  &  \mathbf{E}\int_{0}^{1}\Big[\Big\langle p(t),f^{\Delta,\eps}(t)+\frac{1}
{2}\bar{f}_{xx}(t)\left(  x_{1}^{\eps}\otimes x_{1}^{\eps}\right)
(t)\Big\rangle\Big]\,\mathrm{d}t\\
&  +\mathbf{E}\int_{0}^{1}\Big[\Big\langle q(t),g^{\Delta,\eps}(t)+\frac{1}{2}
\bar{g}_{xx}(t)\left(  x_{1}^{\eps}\otimes x_{1}^{\eps}\right)
(t)+g_{x}^{\Delta,\eps}
(t)x_{1}^{\eps}(t)\Big\rangle\Big]\,\mathrm{d}t\\
=~  &  o(\varepsilon)+\mathbf{E}\int_{0}^{1}\left[
\left\langle p(t),f^{\Delta,\eps}(t)\right\rangle +\left\langle
q(t),g^{\Delta,\eps
}(t)\right\rangle \right]  \,\mathrm{d}t\\
&  +\frac{1}{2}\mathbf{E}\int_{0}^{1}\left[  \left\langle p(t),\bar{f}
_{xx}(t)\left(  x_{1}^{\eps}\otimes x_{1}^{\eps}\right)  (t)\right\rangle
+\left\langle q(t),\bar{g}_{xx}(t)\left(  x_{1}^{\eps}\otimes
x_{1}^{\eps}\right) (t)\right\rangle \right]  \,\mathrm{d}t,
\end{align*}
this along with Lemma \ref{lem:2expansion} yields
\begin{align*}
o(1)~\leq~  &  \varepsilon^{-1}\mathbf{E}\int_{0}^{1}\left[  \mathcal{H}%
(t,\bar{x}(t),u^{\varepsilon}(t),p(t),q(t))-\mathcal{H}(t,\bar{x}(t),\bar
{u}(t),p(t),q(t))\right]  \,\mathrm{d}t\\
&  +\frac{1}{2}\varepsilon^{-1}\mathbf{E}\int_{0}^{1}\left\langle
x_{1}^{\eps}(t),\mathcal{H}_{xx}(t,\bar{x}(t),\bar{u}(t),p(t),q(t))x_{1}^{\eps}%
(t)\right\rangle \,\mathrm{d}t\\
&  +\frac{1}{2}\varepsilon^{-1}\mathbf{E}\left\langle
x_{1}^{\eps}(1),h_{xx}(\bar {x}(1))x_{1}^{\eps}(1)\right\rangle .
\end{align*}
Recalling the definition of $u^{\varepsilon}(\cdot)$, we conclude the lemma.
\end{proof}

\subsection{Second-order duality analysis and completion of the proof}

In this subsection, we deal with the second-order expansion part, i.e., the
second and third terms on the right-hand side of inequality
\eqref{eq:varitionineq}, and complete the proof of Theorem \ref{thm:mp}.

Recall the four-tuple $(\tilde{A},\tilde{B};\tilde{M},\tilde{N})$ with
\begin{align*}
&  \tilde{A}(t) :=A(t)+\bar{f}_{x}(t),~~\tilde{B}(t):=B(t)+\bar{g}_{x}(t),\\
&  \tilde{M} :=h_{xx}(\bar{x}(1)),~~\tilde{N}(t):=\mathcal{H}_{xx}(t,\bar
{x}(t),\bar{u}(t),p(t),q(t)).
\end{align*}
Bearing in mind Assumptions \ref{ass:onAB}--\ref{ass:onB} and the estimate
\eqref{eq:estforajoint}, we can easily obtain that the four-tuple $(\tilde
{A},\tilde{B};\tilde{M},\tilde{N})$ is ``appropriate'', and then from Theorem
\ref{thm:repres} there exists a unique representation 
$$P_{\cdot}\in
L_{\mathrm{w}}^{2}(\Omega\times\lbrack0,1],\cP,\mathfrak{B}(H))
\cap 
C_{\mathbb{F}}([0,1],L_{\mathrm{w}}^{2}(\Omega,\mathfrak{B}(H)))$$ 
with respect
to $T_{\cdot}(\tilde{A} ,\tilde{B};\tilde {M},\tilde{N})$. Therefore, the
assertion (ii) of Theorem \ref{thm:mp} is proved.

From now on we shall let $\tau$ be variable, and, to be more clarified, write
$x^{\tau,\eps}_1=x^{\eps}_1$ and $g^{\Delta,\tau,\eps}=g^{\Delta,\eps}$ to
indicate the dependence of $\tau$.

Now we fix a predictable version $\widetilde{g}^{\Delta}(\cdot)$ of
$g^{\Delta}(\cdot)$ such that
$\mathbf{E}\Vert\widetilde{g}^{\Delta}(\tau)\Vert_{H}^4 < \infty$ for each
$\tau\in [0,1]$. We introduce the following equation
\[
z^{\tau,\varepsilon}(t)=\int_{\tau}^{t}\tilde{A}(s)z^{\tau,\varepsilon}(s)\,\mathrm{d}%
s+\int_{\tau}^{t}[\tilde{B}(s)z^{\tau,\varepsilon}(s)+\varepsilon^{-\frac{1}{2}%
}\widetilde{g}^{\Delta}(\tau){\bf
1}_{[\tau,\tau+\varepsilon]}(s)]\,\mathrm{d}W_{s},\quad t\in[\tau,1].
\]
Then we have

\begin{lemma}
\label{lem:approfx1} For a.e. $\tau\in\lbrack0,1),$%
\[
\lim_{\varepsilon\downarrow0}\mathbf{E}\sup_{t\in\lbrack\tau,1]} \big\Vert
\varepsilon^{-\frac{1}{2}}x^{\tau,\eps}_1(t)-z^{\tau,\varepsilon}(t)\big\Vert_{H}^{4}=0.\
\]
\end{lemma}

\begin{proof}
Recall that
$g^{\Delta,\tau,\eps}(t)=g^{\Delta,\eps}(t)=g^{\Delta}(t)\cdot\mathbf{1}_{[\tau,\tau+\eps]}(t)$.
In view of Lemma \ref{lem:Lpestimate}, the H\"older inequality and the Fubini
theorem, we have that, for each $\tau\in\lbrack0,1)$,
\begin{align*}
\mathbf{E}\sup_{t\in\lbrack\tau,1]}\big\Vert
\varepsilon^{-\frac{1}{2}}x^{\tau,\eps}_1(t)-z^{\tau,\varepsilon}(t)\big\Vert_{H}^{4}
\le \frac{C}{\varepsilon^2}\E \bigg( \int_{\tau }^{\tau+\varepsilon}\left\Vert
\widetilde{g}^{\Delta}(t)-\widetilde{g}^{\Delta} (\tau)\right\Vert
_{H}^{2}\cdot {\bf 1}_{[\tau,\tau+\eps]}(t) \, \d t \bigg)^2
\\
\leq
\frac{C}{\varepsilon}\mathbf{E}\int_{\tau
}^{\tau+\varepsilon}\left\Vert
\widetilde{g}^{\Delta}(t)-\widetilde{g}^{\Delta} (\tau)\right\Vert
_{H}^{4}\,\mathrm{d}t
\leq
\frac{C}{\varepsilon}\int_{\tau
}^{\tau+\varepsilon}\mathbf{E}\left\Vert
\widetilde{g}^{\Delta}(t)-\widetilde{g}^{\Delta} (\tau)\right\Vert
_{H}^{4}\,\mathrm{d}t.
\end{align*}
By the Lebesgue differentiation theorem, we have for each $X\in L^{4}
(\Omega,\mathcal{F}_{1},H)$,
\begin{equation}\label{eq:proof.501}
\lim_{\varepsilon\downarrow0}\frac{1}{\varepsilon}\int_{\tau}^{\tau
+\varepsilon}\mathbf{E}\left\Vert \widetilde{g}^{\Delta}(t)-X\right\Vert _{H}
^{4}\,\mathrm{d}t=\mathbf{E}\left\Vert
\widetilde{g}^{\Delta}(\tau)-X\right\Vert _{H} ^{4},\ \ \ \text{for a.e.
}\tau\in\lbrack0,1).
\end{equation}
Since $L^{4}(\Omega,\mathcal{F}_{1},H)$ is separable, let $X$ run through a
countable dense subset $Q$ of $L^{4}(\Omega,\mathcal{F}_{1},H)$, and denote
\[
E:=\bigcup\nolimits_{X\in Q} E_{X}:=\bigcup\nolimits_{X\in
Q}\big\{\tau:\text{relation \eqref{eq:proof.501} does not hold for }X\big\}.
\]
Then $\mathrm{meas}(E)=0$. For each $\tau\in\lbrack0,1)\backslash E$ and any positive $\eta$, we can take an $X_{\tau,\eta}\in Q$ such
that
\[
\mathbf{E}\left\Vert \widetilde{g}^{\Delta}(\tau)-X_{\tau,\eta}\right\Vert _{H}^{4}<\eta,
\]
then we have
\begin{align*}
&  \lim_{\varepsilon\downarrow0}\frac{1}{\varepsilon}\int_{\tau}
^{\tau+\varepsilon}\mathbf{E}\left\Vert \widetilde{g}^{\Delta}(t)-\widetilde{g}^{\Delta}
(\tau)\right\Vert _{H}^{4}\,\mathrm{d}t\\
&  \leq\lim_{\varepsilon\downarrow0}\frac{8}{\varepsilon}\int_{\tau}
^{\tau+\varepsilon}\mathbf{E}\left\Vert \widetilde{g}^{\Delta}(t)-X_{\tau,\eta}\right\Vert _{H}
^{4}dt+8\mathbf{E}\left\Vert \widetilde{g}^{\Delta}(\tau)-X_{\tau,\eta}\right\Vert _{H}^{4}\\
&  \leq16\mathbf{E}\left\Vert \widetilde{g}^{\Delta}(\tau)-X_{\tau,\eta}\right\Vert
_{H}^{4}<16\eta.
\end{align*}
From the arbitrariness of $\eta$, we conclude this lemma.
\end{proof}

Thanks to the the above lemma, we have
\begin{align}
&  \varepsilon^{-1}\mathbf{E}\int_{0}^{1}\left\langle
x_{1}^{\tau,\eps}(t),\tilde
{N}(t)x_{1}^{\tau,\eps}(t)\right\rangle \,\mathrm{d}t+\varepsilon^{-1}\mathbf{E}%
\left\langle x_{1}^{\tau,\eps}(1),\tilde{M}x_{1}^{\tau,\eps}(1)\right\rangle
\label{eq:simpl2ordervar}\\
&  =o(1)+\mathbf{E}\int_{\tau}^{1}\left\langle z^{\tau,\varepsilon}(t),\tilde
{N}(t)z^{\tau,\varepsilon}(t)\right\rangle \,\mathrm{d}t+\mathbf{E}\left\langle
z^{\tau,\varepsilon}(1),\tilde{M}z^{\tau,\varepsilon}(1)\right\rangle ,\ \
\forall \tau\in\lbrack0,1)\backslash E. \nonumber
\end{align}
Keeping in mind the above relation, and applying Proposition
\ref{prop:propofP}, we conclude for each $\tau\in[0,1)\backslash E$,
\begin{align}\label{eq:proof503}
\mathbf{E}\left\langle
\widetilde{g}^{\Delta}(\tau),{P}_{\tau}\widetilde{g}^{\Delta}(\tau)\right\rangle
=\lim_{\varepsilon\downarrow0}\varepsilon^{-1}\left\{  \mathbf{E}\int_{0}%
^{1}\left\langle x_{1}^{\eps}(t),\tilde{N}(t)x_{1}^{\eps}(t)\right\rangle \,\mathrm{d}%
t+\mathbf{E}\left\langle x_{1}^{\eps}(1),\tilde{M}x_{1}^{\eps}(1)\right\rangle
\right\}.
\end{align}
In view of Lemma \ref{cor:2expansion}, by denoting
\[
\mathcal{H}^{\Delta}(\tau):=\mathcal{H}(\tau,\bar{x}(\tau),u(\tau),p(\tau
),q(\tau))-\mathcal{H}(\tau,\bar{x}(\tau),\bar{u}(\tau),p(\tau),q(\tau )),
\]
and using the Lebesgue differentiation theorem and \eqref{eq:proof503}, we
obtain
\begin{equation}\label{eq:proof502}
0~\leq~ \mathbf{E}\left[ \mathcal{H}^{\Delta}(\tau)+\frac{1}{2}\left\langle
\widetilde{g}^{\Delta}(\tau),{P}_{\tau}\widetilde{g}^{\Delta}(\tau)\right\rangle\right],\
\ \ \text{a.e. }\tau\in\lbrack0,1).
\end{equation}
In view of the Fubini
theorem, we have
\begin{align*}
\int_0^1\mathbf{E}\left\langle
\widetilde{g}^{\Delta}(\tau),{P}_{\tau}\widetilde{g}^{\Delta}(\tau)\right\rangle
\d \tau 
=
\mathbf{E}\int_0^1\left\langle
\widetilde{g}^{\Delta}(\tau),{P}_{\tau}\widetilde{g}^{\Delta}(\tau)\right\rangle
\d \tau 
=
\mathbf{E}\int_0^1\left\langle
{g}^{\Delta}(\tau),{P}_{\tau}{g}^{\Delta}(\tau)\right\rangle \d \tau.
\end{align*}
Combining with \eqref{eq:proof502}, we obtain
\begin{align*}
0~\leq~
\mathbf{E}\int_0^1\Big[\mathcal{H}^{\Delta}(\tau)+\frac{1}{2}\left\langle
{g}^{\Delta}(\tau),{P}_{\tau}{g}^{\Delta}(\tau)\right\rangle\Big] \d \tau.
\end{align*}
Therefore, the desired maximum condition \eqref{eq:mpcondition} follows from
the arbitrary choice of control $u(\cdot)$. This completes the proof of Theorem \ref{thm:mp}.

\section{Examples}

In the following let us discuss two examples which can be covered by the abstract results of the present paper.

\begin{example}[SPDE with controlled coefficients]\rm
Given a bounded domain $\cD\subset \R^n$, we consider the following controlled stochastic PDE
\begin{equation}\label{eq:spde}
  \d \state(t,\xi) = [A\state(t,\xi)+f(t,\xi,u_t)\state(t,\xi)]\,\d t
  + [B\state(t,\xi)+g(t,\xi,u_t)\state(t,\xi)]\,\d W_t
\end{equation}
with $(t,\xi)\in [0,1]\times \cD$, initial data $\state(0,\xi)=\state_0(\xi)$ and a proper boundary condition. Here the control $u_{\cdot}$ is a stochastic process with values in a set $U\subset\R$, and $A,B$ are differential operators
defined as
\[A(t,\xi) = \sum_{i,j=1}^{n}a^{ij}(t,\xi)\frac{\partial^2}{\partial \xi_i \partial \xi_j}, \quad B(t,\xi) =
\sum_{i=1}^{n}\sigma^{i}(t,\xi)\frac{\partial}{\partial \xi_i}.\]
The equation is usually called super-parabolic SPDE if there is a constant $\kappa > 0$ such that 
$$ \kappa I_n \le (2 a^{ij} - \sigma^{i} \sigma^{j})_{n\times n} \le \kappa^{-1} I_n.$$
The fundamental theory about this kind of SPDEs can be found in
\cite{chow2007stochastic,rozovsky1990stochastic}, etc. Here we consider minimizing the cost functional
\begin{equation}\label{eq:J}
  J(u(\cdot))= \E \int_{\cD} \vert \state(t,\xi) \vert^2\, \d \xi
  + \E \int_0^1\!\!\int_{\cD} l(t,\xi,u_t)\vert \state(t,\xi) \vert^2 \, \d \xi \d t.
\end{equation}
Provided some assumptions on coefficients, such as $f,g$ and $l$ are all $\cP\times\mathcal{B}(\cD)\times\mathcal{B}(U)$-measurable functions and dominated by a given constant, one can verify the conditions of Theorem \ref{thm:mp} and apply the abstract result directly in this case.
\end{example}

\begin{example}\rm
Given a bounded domain $\cD\subset \R^n$, we consider the stochastic heat equation
\begin{equation*}\left\{
\begin{aligned}
& \d\state(t,\xi) = \Delta \state(t,\xi)\,\d t + b(\xi) u(t,\xi)\,\d W_t, \quad (t,\xi)\in[0,1]\times \cD;\\
& \state(0,\xi)= \state_0(\xi),\ \xi\in \cD;
\quad \state(t,\xi) = 0,\ (t,\xi)\in[0,1]\times \partial\cD.
\end{aligned}
\right.
\end{equation*}
where the control $u(\cdot,\cdot)$ is a random field with values in a set $E\subset \R$, and the coefficient $b$ is a given bounded function.
The objective of the control problem is to minimize the following cost functional
\begin{equation*}
  J(u(\cdot))= \frac{1}{2}\E \int_{\cD} \vert \state(1,\xi))\vert^2 \, \d \xi
  + \E \int_0^1\!\!\int_{\cD}c(\xi)\vert u(t,\xi) \vert^2 \, \d \xi \d t.
\end{equation*}
This problem can be covered by our result by taking
\[ H = L^2(\cD), \quad U = L^{2}(\cD;E),\quad A(t)\equiv \Delta,\quad B(t)\equiv 0. \]
Let $(\bar{\state},\bar{u})$ be an optimal solution. Then the first-order adjoint process $(p,q)$ is given by the following equation
\begin{equation*}\left\{
\begin{aligned}
& \d p(t,\xi) = - \Delta p(t,\xi)\,\d t + q(t,\xi)\,\d W_t, \quad (t,\xi)\in[0,1]\times \cD;\\
& p(1,x)= \bar{\state}(1,x),\ x\in \cD;
\quad p(t,\xi) = 0,\ (t,\xi)\in[0,1]\times \partial\cD.
\end{aligned}
\right.
\end{equation*}
Moreover, it is easy to verify that the second-order adjoint process $P_{\cdot}$ satisfies 
$$\langle f,P_{t}f \rangle = \int_{\cD}\big\vert\e^{(1-t)\Delta}f\big\vert^2(\xi)\,\d \xi,
\quad \forall\,f\in H$$ 
for each $t\in[0,1]$, where $(e^{t\Delta})_{t\ge 0}$ is the semigroup generated by the Laplacian operator on $H$.
In view of Theorem \ref{thm:mp} we can write down the following necessary condition for optimal control: for any $u(\cdot,\cdot) \in U$,
\begin{align*}
\int_{\cD} \Big\{ b(\xi)q(t,\xi)[u(t,\xi)- \bar{u}(t,\xi)] 
+ c(\xi)[\vert u(t,\xi)\vert^2 - \vert \bar{u}(t,\xi)\vert^2] & \\
+ \frac{1}{2} \big\vert
\e^{(1-t)\Delta}[bu(t,\cdot)-b\bar{u}(t,\cdot)]\big\vert^2(\xi)\Big\} \, \d \xi &~\geq~ 0
\end{align*}
holds for a.e. $(t,\omega)$. 


\end{example}

\begin{remark}\rm
(1) The above examples cannot be covered by the
results in either \cite{lvzhangmaximum} or \cite{fuhrman2012stochastic}
due to the quadratic form of the cost functionals.  

(2) Frankly speaking, the requirement of twice Frech\`et differentiability of coefficients restricts the applicability of our abstract result; for instance, the SPDE with general Nemytskii-type coefficients rarely fits into this framework, and yet the result obtained in \cite{fuhrman2012stochastic} is not covered by ours. Nevertheless, we believe that some key approaches in this paper, especially the techniques used in second-order duality analysis, could also apply to many other concrete problems. Further investigations thereof would be presented in future publications.
\end{remark}

\section*{Acknowledgments}
The authors would like to thank the referees and editor for their helpful
comments and suggestions.

\bibliographystyle{siam}
\bibliography{MPforSEE}

\end{document}